\documentclass[11pt]{amsart}
\usepackage{amsmath,amsthm,amssymb,bm}
\usepackage{color,textcomp}
 \usepackage{indentfirst}
\usepackage{graphicx}
\usepackage{tikz}
\usetikzlibrary{decorations.pathreplacing}
\usetikzlibrary{arrows,calc}
\usepackage{mathrsfs}
\usepackage{esint}
\usepackage{comment}
\usepackage[normalem]{ulem}
\usepackage[top=1in, bottom=1in, left=1.25in, right=1.25in]{geometry}

\newtheorem{thm}{Theorem}[section]
\newtheorem{prop}[thm]{Proposition}
\newtheorem{coro}[thm]{Corollary}
\newtheorem{lem}[thm]{Lemma}
\newtheorem{remark}[thm]{\textit{Remark}}
\newtheorem{defn}[thm]{Definition}

\newcommand{\beq}{\begin{equation}}
\newcommand{\eeq}{\end{equation}}

\newcommand{\ep}{{\epsilon}}

\newcommand{\del}{\delta}

\newcommand{\Om}{\Omega}

\def\R{\mathbb{R}}

\def\bx'{{\bar{\bm x}}}
\def\bby{{\bar{\bm y}}}

\def\bs{{\bf s}}

\def\bx{{\bf x}}
\def\by{{\bf y}}

\newcommand{\bdx}{{\bf x}}
\newcommand{\bdy}{{\bf y}}
\newcommand{\bds}{{\bf s}}
\newcommand{\bdz}{{\bf z}}
\newcommand{\bdw}{{\bf w}}
\def\Nw{\mathfrak{W}^{s,p}_{\vartheta}}
\def\Wf{\mathfrak{W}^{s,p}}
\newcommand{\dist}{\mathrm{dist}}

\everymath{\displaystyle}
\usepackage[backref]{hyperref}
\title[Fractional Hardy-type and trace theorems for nonlocal function spaces]{Fractional Hardy-type and trace theorems for nonlocal function spaces with heterogeneous localization}

\author{Qiang Du}
\address{Department of Applied Physics and Applied Mathematics, Columbia University, New York, NY
10027}
\email{qd2125@columbia.edu}
\author{Tadele Mengesha}
\address{Department of Mathematics, The University of Tennessee, Knoxville, TN 37996}
\email{mengesha@utk.edu}
\author{Xiaochuan Tian}
\address{Department of Mathematics,
University of California, San Diego,  CA 92093}
\email{xctian@ucsd.edu}

\thanks{Qiang Du's research is supported in part by the NSF DMS-2012562,  DMS-1937254
and  ARO MURI Grant W911NF-15-1-0562.
Tadele Mengesha's research is supported by NSF DMS-1910180. Xiaochuan Tian's research is supported in part by NSF DMS-2044945 and DMS-2111608.}

\begin{document}

\begin{abstract}  
 {This work aims to prove a Hardy-type inequality and a trace theorem for a class of
function spaces on smooth domains with a nonlocal character. Functions in these spaces are  {allowed to be} as  rough as   {an $L^p$-function} inside the domain of definition but as smooth as a $W^{s,p}$-function near the boundary. This feature is captured by a norm that  is characterized by  a 
nonlocal interaction kernel defined heterogeneously with a special 
localization feature on the boundary.
Thus, the trace theorem we obtain here can be viewed as an improvement and refinement of the classical trace
theorem for fractional Sobolev spaces $W^{s,p}(\Omega)$.
Similarly, the Hardy-type inequalities we establish for functions  that vanish on the boundary  show that  functions in this generalized space have the same decay rate to the boundary as functions in the smaller space $W^{s,p}(\Om)$. 
The results we prove extend existing results shown in the Hilbert space setting with $p=2$. A Poincar\'e-type inequality we establish for the function space
under consideration  together with the new trace theorem allow formulating and proving well-posedness of
a nonlinear nonlocal variational problem with conventional local boundary condition.}
\end{abstract}



\maketitle

\section{Introduction and Main results}
\label{sec:intro}

In this paper, following \cite{TiDu2017}, we prove a fractional Hardy-type and trace theorems for functions in the space $\mathfrak{W}^{s,p}(\Omega)$  {that} we define as follows. Let $d\geq 2$, $\Omega\subset \mathbb{R}^{d}$ be an open set, bounded or unbounded, and $\partial \Omega$ representing its boundary which is assumed to have sufficient regularity. 
In the event $\Omega=\mathbb{R}^{d}_{+},$ the half space, or $\Omega=\R^{d}_{M^{+}}:=\R^{d-1}\times (0, M]$,  {then the boundary $\partial\Omega$ is $\mathbb{R}^{d-1}\times\{0\}$.}
Given $s\in (0, 1]$, $1\leq p < \infty$, and for any $u\in L^{p}(\Omega)$, let us introduce the notation $|u|_{\mathfrak{W}^{s,p}(\Omega)}$  defined as 
\[
|u|^{p}_{\mathfrak{W}^{s,p}(\Omega)} := \int_{\Omega} \int_{ {B_{\delta(\bdx)} (\bdx)}} \frac{|u(\by) - u({\bx })|^{p}}{|\delta (\bx)| ^{{\mu} }} d\bdy d\bdx
\]
where {$\mu=d+ps$} and $\delta(\bdx) =  {\dist}(\bdx, \partial \Om)$, and  the notation $B_{r} ({\bdz})$ denotes a ball of radius $r$ and  {centered} at $\bdz$.  It is clear that $|u|_{\mathfrak{W}^{s,p}(\Omega)}$ defines a seminorm. 
We take the function space $\mathfrak{W}^{s,p}(\Omega)$ to be the  {completion} of $C^{0, 1}_c(\overline{\Om})$ with respect to the norm 
\[
\|\cdot\|_{\mathfrak{W}^{s,p}(\Omega)}  =\left( \|\cdot\|_{L^{p}}^{p} + |\cdot|_{\mathfrak{W}^{s,p}(\Omega)}^{p}\right)^{1/p}.
\]
We also denote by the space $\mathfrak{\mathring{W}}^{s,p}(\Omega)$ as the  {completion} of $C^{0, 1}_c(\Om)$ with respect to the norm $\|\cdot\|_{\mathfrak{W}^{s,p}(\Omega)}$. In the above, $\overline{\Om}$ is the closure of $\Om$ and the function spaces $C_c^{0, 1}(\overline{\Omega})$ and $C_c^{0,1}(\Omega)$
 are the set of $C^{0,1}$ (Lipschitz) functions with compact support in $\overline{\Omega}$ and $\Omega$ respectively.  If $\Omega$ is bounded, then $C_c^{0, 1}(\overline{\Omega}) = C^{0, 1}(\overline{\Omega})$. Given a set $X$, the notation $C(X)$ represents the set of continuous function defined over $X$.
The function spaces $\mathfrak{W}^{s,p}(\Omega)$ and $\mathfrak{\mathring{W}}^{s,p}(\Omega)$ are Banach spaces with respect to the norm $\|\cdot\|_{\mathfrak{W}^{s,p}(\Omega)}$ 
as can be checked easily  {from their definition.}


 The main goal of this paper is to prove that functions in $\mathfrak{W}^{s,p}(\Omega)$ behave exactly like functions in the fractional Sobolev space $W^{s, p}(\Omega)$ near the boundary.   
 We recall that 
  {the space $W^{s,p}(\Omega)$ consists of all $L^{p}(\Omega)$-function $u$ 
 that has a finite Gagliardo seminorm
 $|u|_{W^{s,p}(\Omega)}=\left\{\int_{\Omega}\int_{\Omega} \frac{|u(\by) - u({\bx })|^{p}}{|\by -\bx| ^{d+ps}}d\by d\bx\right\}^{1/p}$} 
  \cite{gag57}. To be precise, we show that  functions in $\mathfrak{W}^{s,p}(\Omega)$, subject to condition on $s$ and $p$,  have a well defined trace on $\partial \Omega$ and also support a Hardy-type inequality that  {quantifies} the decay rate of functions in $\mathfrak{W}^{s,p}(\Omega)$ that vanish on the boundary. What distinguishes functions in $\mathfrak{W}^{s,p}(\Omega)$ from functions in the Sobolev spaces is that these functions do not necessarily have any smoothness inside $\Omega$ and can be as rough as typical   {$L^{p}$-functions}.  Rather, the defining property of functions in $\mathfrak{W}^{s,p}(\Omega)$ comes from their regular behavior near the boundary $\partial \Omega$ and our main result of the paper captures that regularity. 

 The motivation for this line of research comes from the prevalence of nonlocal modeling which has become suitable to describe singular and discontinuous behavior in diffusion, image processing, mechanics of materials  and other application areas. 
Typically, nonlocal models are based on integration as opposed to differentiation so as to demand less regularity of their solutions \cite{AMRT10,Du19,DGLZ12SR,DGLZ13,GiOs08,KlSo05,LZOB10,MeDu14,MeDu15,MeKl00}.  As a consequence, the associated function spaces require less smoothness and exhibit a nonlocal character \cite{BBM01,CaSi07,ponce2004,ros15}. In some  applications, one is interested in modeling singular behavior inside the domain subject to some boundary conditions.  This is the case for peridynamics \cite{Silling00}, for example, where the interest is in modeling crack formation and fracture in deforming materials subject to some loading conditions on the lower dimensional boundary. 
We refer to \cite{SLS14} for a computational peridynamics model that uses  position-dependent interaction kernels. Similarly, coupling of nonlocal models with local models through co-dimensional one interfaces through heterogeneous localization has also been proposed \cite{TTD19}, These models naturally need to work with function spaces whose elements have well defined traces on the boundary and/or interfaces but also capable of capturing singularities inside the domain. 
Interest in this aspect of appropriately defined function spaces has  {been increasing} recently.  For example, recent works \cite{TiDu2017,Leoni-trace, Moritz-trace, Foss-trace,Foss-poincare,NNS16} deal with defining a bounded trace operator for functions that may not have classical differential regularity in the interior of the domain of definition. 
 We note that the study of  
function spaces with variable order of smoothness and growth is a popular subject with a rich history and significant recent interest;
see, for instance, \cite{Nakai10,Di04}.

The results shown in this paper are extensions of the study initiated in \cite{TiDu2017} and are in parallel to the well known trace theorems for classical function spaces, see \cite{Demengel,Le09} for  a reference  {on the latter}. 
In particular,  
for  {$\Om\subset \R^d$} with sufficiently smooth boundary $\partial\Om$, 
the trace operator $ {T}$ on $\partial \Om$ 
\[
 {T}u= u |_{\partial \Omega}\qquad \forall u \in C_c^{0,1}(\overline{\Om})\,,
\] 
can be uniquely  extended continuously as a map
from  $W^{s,p}(\Om)$,
the fractional Sobolev space of order $s \in (0, 1]$  and $1< p < \infty$ to the fractional Sobolev space 
$W^{s-\frac{1}{p}, p}(\partial \Om)$ on the boundary of $\Omega$, provided that $ps >1$. In addition,  the trace of a function in $W^{s,p}(\Om)$ is well defined over any smooth hypersurface contained in the domain $\Omega$. Functions in $W^{s,p}(\Om)$ are mildly regular and for that matter, in the event $ps > d$, they are H\"older continuous (\cite{Nezza}), which makes them less useful in the 
modeling   {of} singular behavior inside the domain
 {$\Om$}.    
Similar to the spaces first proposed in \cite{TiDu2017} (for the case of $p=2$), the function space  $\mathfrak{W}^{s,p}(\Omega)$, defined above,  combines the best of $W^{s,p}(\Om)$ on the boundary $\partial \Om$ and $L^{p}(\Omega)$ inside $\Omega$. The main result of this paper is to show that when $\Omega$ is a domain with sufficiently smooth boundary, the trace map exists  and 
is continuous from  $\mathfrak{W}^{s,p}(\Omega)$ to $W^{s-\frac{1}{p}, p }(\partial \Om)$ provided $sp > 1$, while the functions in $\mathfrak{W}^{s,p}(\Omega)$ remain as singular as typical  {$L^{p}$-functions} in the interior of $\Omega$.  

\begin{thm}[General trace theorem]\label{coro:bdddomain-intro}

Let $1\leq p < \infty$, $s\in (0, 1]$ and $ps > 1$. Assume that $\Omega$ is a bounded Lipschitz domain  in $\mathbb{R}^{d}(d\geq 2)$. Then the linear operator 
\[
  {T}u = u|_{\partial \Omega}, \quad u\in C^{0,1}(\overline{\Omega})
\]
has a unique extension to a bounded linear operator 
\[
 {T}: \Wf(\Omega) \to W^{s- \frac{1}{p}, p}(\partial \Omega), 
\]
and there exists a constant $C$ depending only on $s, p, d $ and the boundary character of $\Omega$ such that 
\begin{equation} \label{trace-estimate}
\| {T}u\|_{W^{s-\frac{1}{p},p}(\partial\Omega)} \leq C \|u\|_{\Wf (\Omega)}, \quad \forall u\in \Wf (\Omega).
\end{equation}
\end{thm}

The proof of Theorem \ref{coro:bdddomain-intro} follows the same general procedures adopted in \cite{TiDu2017}, which are standard for the proof of trace theorems for Sobolev spaces, \cite{Di96}. We first establish the validity of the statement over the half-space $\R^{d}_+=\R^{d-1}\times (0, \infty)$, which is the main step that involves significantly different analysis from the traditional case. We then use the partition of unity to extend the result for bounded domains with Lipschitz boundary. 
The proof of the trace theorem for the half space, similar to the case of \cite{TiDu2017}, 
relies on a  {one-dimensional nonlocal fractional Hardy-type inequality, which is a special case of the general Hardy-type inequality stated below.  }  
 
 {
\begin{thm}[Hardy-type inequality]\label{NLHardy-Lips}
 Let $1< p < \infty$, $s\in (0, 1]$ and $ps > 1$. Assume that $\Omega$ is a bounded Lipschitz domain in $\mathbb{R}^{d} \,(d\geq 2)$.
Then there exists a constant $C>0$ such that for any $u\in \mathfrak{\mathring{W}}^{p,s}(\Omega)$
\beq \label{prop:ghardy-semi}
 \int_\Om \dfrac{|u(\bm x)|^p}{( {\dist}(\bm x, \partial \Om))^{ps}} d\bm x \leq C  |u|^p_{\Wf (\Omega)}\,.
 \eeq
 \end{thm}
 }

Theorem \ref{NLHardy-Lips} quantifies the rate of decay to zero for functions with vanishing trace which is more heavily determined by their smoothness near $\partial \Omega$ as they can be rough away from the boundary.  {  Inequalities similar to \eqref{prop:ghardy-semi} have been shown in  \cite[Theorem 4]{Dyda-Vah-Hardy} where a general framework of  fractional Hardy-type inequalities is provided. 
Our proof of the inequality \eqref{prop:ghardy-semi} involves first establishing it for the half-space $\R^d_+$, which is discussed in details in Section \ref{N-Hardy}, and then extending it to general bounded Lipschitz domains by the partition of unity technique.} 


To demonstrate  {potential applications} of the function space in allowing  singularity as part of the solution, we study the problem of minimizing the energy 
\begin{equation}\label{Energy}
E(u) =\int_{\Omega}\int_{ { B_{\delta(\bx)}(\bx)}} \frac{A(\bx, \by)}{\delta(\bx)^{\mu}} {F(u(\by)-u(\bx))}d\bx d\by
\end{equation}
over an appropriate subsets of $\mathfrak{W}^{s, p}(\Omega)$ where $A(\bx, \by)$ serves as a coefficient and is symmetric, and elliptic in the sense that $0<\alpha_1 \leq A(\bx, \by) \leq \alpha_2 < \infty$ for all  {$\bx,\by \in \Omega$. The} function $F:\mathbb{R} \to \mathbb{R}$ is a convex function such that for some positive constants $c_1, c_2$ and $c_3$ 
\begin{equation}\label{Convexity-of-F}
c_1|t|^p \leq F(t) \leq c_{2} |t|^p\text{ and $|F'(t)| \leq c_3|t|^{p-1} $ for almost every $t\in \mathbb{R}$.}
\end{equation}
 {Notice that convex functions are differentiable almost everywhere (see e.g. \cite[Theorem 25.5]{rockafellar1970}).}
To define the subset over which we minimize the energy, for $1<p<\infty $ and $s\in (0, 1)$ such that $ps> 1$,  {We consider some prescribed data given by} $\phi\in W^{s, p}(\Omega)$  {and} $0\leq h\in L^{p}(\Omega)\setminus {W}^{s, p}(\Omega)$. For a fixed open subset $\Omega_0\Subset \Omega$ (compactly contained), consider the set of functions 
\[
K_{\phi}(p, h)=\{u\in \mathfrak{W}^{s, p}(\Omega): u\geq h,\,\text{in $\Omega_0$  }\text{and  } u-\phi\in \mathfrak{\mathring{W}}^{s, p}(\Omega) \}. 
\]
The set $K_{\phi}(p, h)$ collects all functions in  $\mathfrak{W}^{s, p}(\Omega)$ that are above the nonnegative function $h(\bx)$ for $\bx\in\Omega_0$ 
and have the same trace on the boundary  {$\partial\Om$} as $\phi$.
Notice that the function $h=h(\bx)$ is assumed to be a generic  {$L^p$-function}, thus allowing representations of more rough obstacles. 
While behaving like  $W^{s,p}(\Omega)$-functions at the boundary, functions in $K_\phi(p, h)$   {are allowed to} essentially retain the  {potential} singularity of an 
 {$L^p$-function} inside $\Omega_0.$   
 {The set $K_{\phi}(p,h)$ is nonempty as the function $\bx\mapsto \phi(\bx)\chi_{\Omega\setminus \Omega_0}(\bx)+ h(\bx)\chi_{\Omega_0}(\bx)\in K_{\phi}(p,h).$ Indeed, since $\phi\in W^{s,p}(\Omega)$, it suffices to show that $u-\phi= (h-\phi) \chi_{\Omega_0}\in \mathfrak{\mathring{W}}^{s, p}(\Omega)$. To that end, let $\varphi_{n} \in C_c^{0, 1}(\Omega_0)$ such that $\|h-\phi-\varphi_{n}\|_{L^{p}(\Omega_0)}\to0$ as $n\to\infty$. Now there exists $\delta_0 >0$ such that $\text{dist}(\bx, \partial \Omega) \geq \delta_0$ for all $\bx\in \Omega_0$ and so for some positive constant $C$,
\[
\|h-\phi-\varphi_{n}\|_{\mathfrak{W}^{s, p}(\Omega)} \leq C  \delta_0^{-\frac{d+ps}{p}} \|h-\phi-\varphi_{n}\|_{L^{p}(\Omega_0)} \to 0, \quad \text{as $n\to \infty$,}
\] where extending by 0 outside $\Omega_0$, we assume that $\varphi_{n} \in C_c^{0, 1}(\Omega)$. 
It is also clear that the set  $K_{\phi}(p,h)$ is a convex and closed subset of $\mathfrak{W}^{s, p}(\Omega)$. Notice that by the trace theorem, since $ps> 1$, any function in $K_{\phi}(p,h)$ has a well defined trace 
that 
agrees with that of the trace of $\phi$.} 
 {Another ingredient we need for the coercicity of the energy $E$ is the nonlocal Poincar\'e-type inequality that will be established in Proposition \ref{Poincare}. The Poincar\'e-type inequality  extends a result stated in \cite{TTD19} and also fills in a gap in the proof presented there. In particular, special treatment near the boundary of $\Omega$ is needed for the type of nonlocal kernel with heterogeneous localization on the boundary.}
As an application of the direct method of calculus of variations,  {using the Poincar\'e-type inequality,} we have the following existence result.
\begin{thm}\label{obstacle}
 {Let $\Omega\subset\mathbb{R}^d$ be a bounded Lipschitz domain.}
For $1<p<\infty $ and $s\in (0, 1)$ such that $ps> 1$, fix $\phi\in W^{s, p}(\Omega)$ and $0\leq h\in L^{p}(\Omega)\setminus {W}^{s, p}(\Omega)$
and $f\in 
[\mathfrak{W}^{s, p}(\Omega)]^\ast$, the dual space of $\mathfrak{W}^{s, p}(\Omega)$.
Then, there is a $u\in K_{\phi}(p, h)$
such that 
\[
E(u) - \langle f, u\rangle = \inf_{v\in K_{\phi}(p, h)}\left(E(v) - \langle f,  {v}\rangle\right).
\]
where $E$ is the energy functional defined in \eqref{Energy} and $\langle f, u\rangle$ is the action of $f$ on $u$. 
\end{thm}
{In the special case of $p=2$ and $F(t)=t^2$, the minimizer solves a linear variational inequality.}
\begin{coro}\label{obstaclep=2}
 {Let $\Omega\subset\mathbb{R}^d$ be a bounded Lipschitz domain.}
For $p=2$, let $s\in (0, 1)$ such that $2s> 1$, fix $\phi\in W^{s, 2}(\Omega)$, $0\leq h\in L^{2}(\Omega)\setminus {W}^{s, 2}(\Omega)$ and
 $f\in 
 [\mathfrak{W}^{s, 2}(\Omega)]^\ast 
 $. Then, 
 there is a unique minimizer 
 $u\in K_{\phi}(2, h)$ defined by
 \[
E(u) - \langle f, u\rangle = \inf_{v\in K_{\phi}(2, h)}\left(E(v) - \langle f, u\rangle\right)
\]
where $E$ is the energy functional defined in \eqref{Energy} with $p=2$.
Moreover, $u$ satisfies the variational inequality 
\begin{equation}\label{obstacle-inequality}
B(u, v-u)\geq  \langle f, v-u\rangle,\quad \forall v\in K_{\phi}(2, h),
\end{equation}
{for the bilinear form $B(u, v)$ defined by 
\begin{equation}\label{bilinearform}
B(u, v) = \frac{1}{2}\int_{\Omega}\int_{\Omega} G(\bx, \by)(u(\bx)-u(\by))(v(\bx)-v(\by))d\by d\bx,\quad \text{for all $u,v\in \mathfrak{W}^{s, 2}(\Omega).$ } 
\end{equation}
and
\[
G(\bx, \by) =A(\bx, \by)\left(  \frac{\chi_{B_{\delta(\bx)} {(\bm 0)}} (\bx-\by)}{\delta(\bx)^{d+2s}}+   \frac{\chi_{B_{\delta(\by)}  {(\bm 0)}}(\bx-\by)}{\delta(\by)^{d+2s}}\right).
\]
}
\end{coro}


This paper is organized as follows. In Section \ref{Prelim}, we present some preliminary results collecting estimates that will be used through out the paper. Some properties of the space $\mathfrak{W}^{s,p}(\Omega)$ will be proved. In Section \ref{N-Hardy}, the nonlocal Hardy-type inequality 
 for a stripe and half space will be proved. This result will be used in Section \ref{trace-strip} to establish the trace theorem for functions defined on stripes and a half space. The main theorem on the trace of functions on a bounded  Lipschitz domain is proved in Section \ref{trace-bounded} {together with a proof of the nonlocal Hardy-type inequality and Poincar\'e-type inequality on such general domains}. The proofs of Theorem \ref{obstacle} and Corollary \ref{obstaclep=2} are presented in Section \ref{appl}. Finally, concluding remarks are given at the end.

\section{Some definitions and preliminary estimates}\label{Prelim}
In this section we present some elementary estimates that will be frequently cited throughout the paper. 
The proof of the trace theorem also requires us to introduce the parametrized space of functions $\mathfrak{W}^{s,p}_{\vartheta}(\Omega)$. Given $s\in (0, 1]$, $1\leq p < \infty$, a number $\vartheta \in (0, 1]$ and for any $u\in L^{p}(\Omega)$, let us introduce the notation $|u|_{\mathfrak{W}^{s,p}_{\vartheta}(\Omega)}$  defined as 
\[
|u|^{p}_{\mathfrak{W}^{s,p}_{\vartheta}(\Omega)} := \int_{\Omega} \int_{ { B_{\vartheta\delta(\bx)}(\bx)}} \frac{|u(\by) - u({\bx })|^{p}}{|\vartheta\delta (\bx)|^{{\mu}}} d\bdy d\bdx
\]
where $\delta(\bdx) =  {\dist}(\bdx, \partial \Om)$. As before,  if $\Omega$ is the special unbounded sets, either the half space $\mathbb{R}^{d}_{+},$ or the stripe  $\R^{d}_{M^{+}}:=\R^{d-1}\times (0, M]$, then the boundary $\partial\Omega$ is  $\mathbb{R}^{d-1}\times\{0\}$ and $\delta(\bdx', x_d)= x_d$. It is clear that $|u|_{\mathfrak{W}^{s,p}_{\vartheta}(\Omega)}$ defines a seminorm. 
We  now define the function space $\mathfrak{W}^{s,p}_{\vartheta}(\Omega)$ to be the  {completion} of $C_c^{0,1}(\overline{\Om})$ with respect to the norm 
\[
\|\cdot\|_{\mathfrak{W}^{s,p}_{\vartheta}(\Omega)}  = \left( \|\cdot\|_{L^{p}}^{p} + |\cdot|_{\mathfrak{W}^{s,p}_{\vartheta}(\Omega)}^{p}\right)^{1/p}. 
\]
We observe that if $|u|^{p}_{\mathfrak{W}^{s,p}_{\vartheta}(\Omega)} =0$, then $u$ is a constant  {on any connected component of $\Omega$}. Indeed, for  any $D$ compactly contained in  {a connected component of} $\Omega,$ for any $\bx\in D,$ $\delta(\bx) \geq  {\dist}(D,\partial \Omega)=\varepsilon$. Then we have 
\[
0=|u|^{p}_{\mathfrak{W}^{s,p}_{\vartheta}(\Omega)} \geq  \int_{D} \int_{D\cap B_{\vartheta\varepsilon}(\bx)}  \frac{|u(\by) - u({\bx })|^{p}}{|\vartheta\delta (\bx)| ^{{\mu} }} d\bdy d\bdx
\]
from which we have $u(\by) = u({\bx })$ for all $\by\in B_{\vartheta\varepsilon}(\bx)$. By covering a chain of intersecting balls, we have that $u$ is constant in $D$ and therefore constant in  {the connected component of $\Omega$ to which $D$ belongs}.  
By introducing the symmetrized kernel 
\begin{equation}\label{symmetrized-kernel}
\gamma^\vartheta_{p}(\bx, \by) =  \frac{\chi_{B_{\vartheta\delta(\bx)} {(\bm 0)}} (\bx-\by)}{|\vartheta\delta(\bx)|^{{\mu} }}+   \frac{\chi_{B_{\vartheta\delta(\by)}  {(\bm 0)}}(\bx-\by)}{|\vartheta\delta(\by)|^{{\mu} }}\end{equation} we can write the seminorm as 
\[
|u|^{p}_{\mathfrak{W}^{s,p}_{\vartheta}(\Omega)} = \frac{1}{2} \int_{\Omega} \int_{\Omega}  \gamma^\vartheta_{p}(\bx, \by) |u(\by) - u({\bx })|^{p}
d\bdy d\bdx. 
\]
For $\vartheta=1,$ we use the notation $\gamma_{p}$ and $\mathfrak{W}^{s,p}(\Omega)$ instead of $\gamma^{1}_{p}$ and  $\mathfrak{W}^{s,p}_{1}(\Omega)$ respectively.  {By definition, the function space $\mathfrak{W}^{s,p}_{\vartheta}(\Omega)$ is a Banach space  with respect to the norm $\|\cdot\|_{\mathfrak{W}^{s,p}_{\vartheta}(\Omega)}$. }
In addition, using the isometric mapping  
\[
\begin{split}
\mathcal{G}: &\mathfrak{W}^{s,p}_{\vartheta}(\Omega) \to L^{p}(\Omega)\times L^{p}(\Omega\times \Omega)\\
&\mathcal{G}u = (u, (u(\by)-u(\bx))\sqrt[p]{\gamma^\vartheta_{p}(\bx, \by)})
\end{split}
\]
it follows that for $1<p<\infty$, $\mathfrak{W}^{s,p}_{\vartheta}(\Omega)$ is a reflexive Banach space
 {\cite[Proposition 3.20]{brezisbook}. See the proof of \cite[Proposition 8.1]{brezisbook} for the argument.}  
In particular for $p=2$, it is a Hilbert space with the inner product 
\[
\langle u, v\rangle = \int_{\Omega} u(\bx)v(\bx)d\bx + \int_{\Omega}\int_{\Omega} \gamma^\vartheta_{2}(\bx-\by) (u(\by) - u({\bx })) (v(\by) - v({\bx })d\by d\bx.
\]
 Notice that since $|\bx-\by|\leq \vartheta \delta(\bx)$ for all $y\in \Omega\cap B_{\vartheta \delta(\bx)} (\bx)$, we have  $|\vartheta \delta(\bx)|^{-(d+ps)} \leq |\bx-\by|^{-(d+ps)}$ and as a consequence, $W^{s,p}(\Omega)\subset \mathfrak{W}^{s,p}_{\vartheta}(\Omega)$ for any $0<s<1$, $p\geq 1$ and $\vartheta\in (0, 1]$. 

We will prove some properties of the space $\mathfrak{W}^{s,p}_{\vartheta}(\Omega)$ that we need. 
\begin{lem}\label{multibysmoothfn}
Let $1\leq p < \infty$ and $s\in (0, 1]$. 
Suppose that  $\Omega$ is  {an open set} with smooth boundary and $\psi\in  {C^{0, 1}_c(\overline{\Omega})}$, $0\leq \psi \leq 1$. Then if $u\in \Wf_{\vartheta}(\Omega)$, then $\psi u\in \Wf_{\vartheta}(\Omega)$ with the estimate
\[
\|u\psi\|_{\Wf_{\vartheta}(\Omega)}\leq C \|u\|_{\Wf_{\vartheta}(\Omega)}
\]
where $C$ depends on $\vartheta, s$, $p$, and $\psi$.  
\end{lem}
\begin{proof}
It suffices to estimate the seminorm $[\psi u]_{\Wf_{\vartheta}(\Omega)}$  {for $u\in C^{0,1}_c(\overline{\Omega})$}.
 {Notice that $\psi$ is bounded and Lipschitz, then we have
\[
\begin{split}
|\psi(\bdy)u(\bdy) - \psi(\bdx)u(\bdx)| &\leq |\psi(\by) (u(\by)-u(\bx))|+ |u(\bx) (\psi(\by) -\psi(\bx))|, \\
&\leq \| \psi \|_{C^{0,1}_c(\overline{\Omega})} \left( |u(\by)-u(\bx)| +  |u(\bx)| (1\wedge |\by-\bx|) \right)
\end{split}
\]
similar to \cite[Lemma 3.64 (ii)]{Gounoue}. 
Using the above estimate, we have 
\[
\begin{split}
&|\psi u|_{\Wf_{\vartheta}(\Omega)}^{p}  = \int_{\Omega}  \int_{B_{\vartheta\delta(\bdx)} (\bdx)}\frac{ |\psi(\bdy)u(\bdy) - \psi(\bdx)u(\bdx)|^{p}}{(\vartheta \delta(\bdx))^{\mu}} d\bdy d\bdx\\
&\leq 2^{p-1}\| \psi \|^p_{C^{0,1}_c(\overline{\Omega})} \left(\int_{\Omega}  \int_{B_{\vartheta\delta(\bdx)} (\bdx)}\frac{ |u(\bdy) - u(\bdx)|^{p}}{(\vartheta\delta(\bdx))^{\mu}} d\bdy d\bdx 
+  \int_{\Omega}  \int_{B_{\vartheta\delta(\bdx)} (\bdx)}\frac{ |u(\bx)|^{p} (1\wedge |\by-\bx|)^p}{(\vartheta\delta(\bdx))^{\mu}} d\bdy d\bdx\right)\\
&\leq 2^{p-1}\| \psi \|^p_{C^{0,1}_c(\overline{\Omega})} \left(|u|^p_{\Wf_{\vartheta}(\Omega)} + \int_{\Omega}  \int_{B_{\vartheta\delta(\bdx)} (\bdx)}\frac{ |u(\bx)|^{p} (1\wedge \vartheta\del(\bx))^p}{(\vartheta\delta(\bdx))^{\mu} } d\bdy d\bdx.\right)\\
&\leq C \left( |u|_{\Wf_{\vartheta}(\Omega)}  + \int_{\Omega}|u(\bdx)|^{p} \left(\delta(\bx)^{-sp}\wedge\delta(\bdx)^{(1-s)p}\right)d\bdx.  \right) \end{split} 
\]
}
for some constant $C$ that depends on $\vartheta, p, s$ and $\psi$.  Notice the last term is bounded by $\|u\|_{L^{p}}$ since $0<s\leq 1$.
This completes the proof.  
\end{proof}

The following technical lemma will be very useful in the  {proofs of the} Hardy-type inequality  {and} the trace theorem for functions defined on stripes. It quantifies the continuous embedding of one space into another and gives a precise comparison estimate between seminorms of the parametrized nonlocal spaces. Moreover, the lemma specifies the constants involved because we will need it later to clearly identify 
 {how the constant $C$ in Theorem \ref{hardy-for-smooth} depends on $\vartheta$}. 

\begin{lem}\label{lemma-half-to-vartheta}
Let $p\geq 1$, $d\geq 1$,  $s\in (0, 1]$ and $ps \geq 1$.   Then for any  any $M\in (0, \infty]$,  for any  $\theta_{0} \in \left(0, 1\right)$ and  any $\vartheta\in \left(0, \frac{\theta_{0}}{1-\theta_{0}}\right]$, 
we have $\mathfrak{W}^{s,p}_{ {\vartheta}}(\mathbb{R}^{d}_{M^+})\subseteq \mathfrak{W}^{s,p}_{\theta_0}(\mathbb{R}^{d}_{M^+})$. Moreover, if 
$u\in 
\mathfrak{W}^{s,p}_{\theta}(\mathbb{R}^{d}_{M^+})$, then 
\[
\begin{split}
{|u|^{p}_{\mathfrak{W}^{s,p}_{\theta_0}(\mathbb{R}^{d}_{M^+})}}
& \leq 2^{p}|B_{1}(\bm 0)|{\theta_{0}^{1-\mu}} [(1- \theta_{0})M]^{-ps} \|u\|^{p}_{L^{p}(  \mathbb{R}^{d}_{M^+} )} \\
 &\quad + 2^{d+p}(1 + \theta_{0})^{\mu}{\theta_{0}^{-\mu}} \left( {\frac{\theta_{0}}{1-\theta_{0}}}\right)^{d-p} \vartheta^{ps-p} {|u|^{p}_{\mathfrak{W}^{s,p}_{\vartheta}(\mathbb{R}^{d}_{M^+})}}. 
\end{split}
\]
\end{lem}
\begin{proof} 
 {Given $\theta_{0} \in \left(0, 1\right)$, $p\geq 1$ and $d\geq 1$ and
 $u\in 
\mathfrak{W}^{s,p}_{\theta_0}(\mathbb{R}^{d}_{M^+})$,} We may   {first consider} $u\in C^{0,1}_c(\overline{\mathbb{R}^{d}_{M^+})}$ so that all  {the norms and seminorm in the above inequality} 
are finite.  {Then by the density argument, we can} extend the  {desired result to $\mathfrak{W}^{s,p}_{\theta}(\mathbb{R}^{d}_{M^+})$.}  To that end, we begin by rewriting the integral on the left hand side as 
\small{\begin{equation}\label{sum-integral}
\begin{split}
\theta_0^\mu|u|^{p}_{\mathfrak{W}^{s,p}_{\theta_0}(\mathbb{R}^{d}_{M^+})}
&= \int_{\mathbb{R}^{d}_{(1-\theta_{0})M^+}}\int_{B_{\theta_{0} x_{d}}  (\bdx)\cap  \mathbb{R}^{d}_{M^+}} \frac{|u(\bdy) - u(\bdx)|^{p} } {|x_{d}|^{\mu}}d\bdy d\bdx\\  &+ \int_{\mathbb{R}^{d-1}}\int_{(1-\theta_{0})M}^{M} \int_{B_{\theta_{0} x_{d} } (\bdx)\cap \mathbb{R}^{d}_{M^+}} \frac{|u(\bdy) - u(\bdx)|^{p} } {|x_{d}|^{\mu}}d\bdy d \bdx\\ 
\end{split}
\end{equation}}We will estimate the two terms in the right hand side of \eqref{sum-integral}.  {We will start with the second term.} Using the change of variables $\bdy = \theta_{0}x_{d}\bdz + \bdx$ we may write as 
\[
\begin{split}
&\int_{\mathbb{R}^{d-1}}\int_{(1-\theta_{0})M}^{M} \int_{B_{\theta_{0} x_{d} } (\bdx)} \chi_{\mathbb{R}^{d}_{M^+}}(\bdy)\frac{|u(\bdy) - u(\bdx)|^{p} } {|x_{d}|^{\mu}}d\bdy d \bdx\\
 &= \int_{\mathbb{R}^{d-1}}\int_{(1-\theta_{0})M}^{M} \int_{B_{1 } (\bm 0)} \chi_{\mathbb{R}^{d}_{M^+}}(\theta_{0} x_{d} \bdz + \bdx)\frac{|u(\theta_{0} x_{d} \bdz + \bdx) - u(\bdx)|^{p} } {|x_{d}|^{\mu}}(\theta_{0} x_{d})^{d} d\bdz d \bdx\\
&\leq \theta_{0}^{d} [(1- \theta_{0})M]^{-ps} \int_{B_{1 } (\bm 0)} \int_{\mathbb{R}^{d-1}}\int_{(1-\theta_{0})M}^{M}\chi_{\mathbb{R}^{d}_{M^+}}(\theta_{0} x_{d} \bdz + \bdx) |u(\theta_{0} x_{d} \bdz + \bdx) - u(\bdx)|^{p} d\bdx d \bdz\\
&\leq 2^{p}|B_{1}(\bm0)| [(1- \theta_{0})M]^{-ps}  \|u\|_{L^{p}(  \mathbb{R}^{d}_{M^+} )}^{p}\\
\end{split}
\]
where we have applied Fubini's theorem. 
We now  {estimate} the first term of the right hand side of \eqref{sum-integral}. Let us denote it by $I(\theta_{0})$ and write it in a slightly different way as 
\[
 I(\theta_{0})
 = \int_{\mathbb{R}^{d}_{(1-\theta_{0})M^+}} \int_{ {B_{\theta_{0} x_{d}}(\bm 0)} } \frac{|u(\bdx + {\bds}) - u(\bdx)|^{p} } {|x_{d}|^{\mu}}d\bds d\bdx . 
\]
 {Notice that for any $\bx\in \mathbb{R}^{d}_{(1-\theta_{0})M^+}$ and $\bs \in B_{\theta_{0} x_{d}}(\bm 0)$, we have $\bx+\bs \in\mathbb{R}^{d}_{M^+}$.}
Now for each $\bds \in  {B_{\theta_{0} x_{d}}(\bm 0)} $ and $n\in \mathbb{N}$ we  write the difference as the telescoping sum of differences  {given by}
\[
u(\bdx + \bds) - u(\bdx) = \sum_{i=1}^{n}
 {[} u(\bdx + \frac{i}{n} \bds) - u(\bdx + \frac{i-1}{n}\bds) 
{]= \sum_{i=1}^{n} [u(\bdx_{i+1})-u(\bdx_i)]
}
\]
 {where $\{\bdx_{i} = \bdx + \frac{i-1}{n}\bds\}_{i=1}^{n+1}$.}
Since $\left(\sum_{i=1}^{n} |a_{i}|\right)^{p} \leq n^{p-1}\sum_{i=1}^{n}|a_{i}|^{p}$ holds for all $1 \leq p <\infty $, we obtain that
\[
\begin{split}
I(\theta_{0})&\leq n^{p-1} \sum_{i=1}^{n} \int_{\mathbb{R}^{d}_{(1-\theta_{0})M^+}} \int_{ {B_{\theta_{0} x_{d}}(\bm 0)} } \frac{|u(\bdx + \frac{i}{n}\bds) - u(\bdx + \frac{i-1}{n}\bds)|^{p} } {|x_{d}|^{\mu}}d\bds d\bdx\\
&\leq n^{p-1}  \sum_{i=1}^{n} \int_{\mathbb{R}^{d}_{(1-\theta_{0})M^+}} \int_{ {B_{\theta_{0} x_{d}}(\bm 0)} } \frac{|u(\bdx_{i} + \frac{1}{n}{\bds}) - u(\bdx_{i} )|^{p} } {|x_{d}|^{\mu}}d\bds d\bdx,
\end{split}
\]
Notice that the $d$-component $[\bdx_{i}]_{d}$ satisfies the inequality that $
(1 -\theta_{0})x_{d} \leq [\bdx_{i}]_{d} \leq (1+\theta_{0} )x_{d}$  {for $i=1,...,n$,} therefore we have 
\[
\begin{split}
I(\theta_{0})&\leq n^{p-1}  \sum_{i=1}^{n} \int_{\mathbb{R}^{d}_{(1-\theta_{0})M^+}} \int_{ {B_{\theta_{0} x_{d}}(\bm 0)} } \frac{|u(\bdx_{i} + \frac{1}{n}{\bds}) - u(\bdx_{i} )|^{p} } {|x_{d}|^{\mu}}d\bds d\bdx
\\
&\leq n^{p-1}  (1 + \theta_{0})^{\mu}\sum_{i=1}^{n} \int_{\mathbb{R}^{d}_{(1-\theta_{0})M^+}} \int_{|\bds| \leq \frac{\theta_{0}}{1-\theta_{0}} [\bdx_{i}]_{d}} \frac{|u(\bdx_{i} + \frac{1}{n}{\bds}) - u(\bdx_{i} )|^{p} } {|[\bdx_{i}]_{d}|^{\mu}}d\bds d\bdx {.}
\end{split}
\]
 {Since $\bx\in \mathbb{R}^{d}_{(1-\theta_0)M^+}$ and $|\bs|\leq \theta_0 x_d $, we have
\[
[\bx_i]_d \leq (1-\theta_0)M + \frac{n-1}{n}\theta_0 (1-\theta_0)M =(1-\theta_0)(1+(n-1)\frac{\theta_0}{n})M =:\beta_n M.  
\]
for each $i=1,\cdots, n$. Notice that $\beta_n=(1-\theta_0)(1+(n-1)\theta_0/n)<(1-\theta_0)(1+\theta_0)<1$.
}
Making the change of variable $\bdx_{i} \to \bdx $ in the outer integral we obtain that 
\[
\begin{split}
I(\theta_{0})&\leq n^{p}  (1 + \theta_{0})^{\mu}\int_{ {\mathbb{R}^{d}_{\beta_n M^+}}} \int_{|\bds| \leq  \frac{\theta_{0}}{1-\theta_{0}}x_{d}} \frac{|u(\bdx + \frac{1}{n}{\bds}) - u(\bdx )|^{p} } {|x_{d}|^{\mu}}d\bds d\bdx.
\end{split}
\]
By a change of variables $\frac{\bds}{n} \to \bds$ in the inner integral and letting {$\theta_{n} = \frac{\theta_{0}}{n(1-\theta_{0})}$}, we get that
\[
\begin{split}
I(\theta_{0})&\leq n^{d +p}  (1 + \theta_{0})^{\mu} \int_{ {\mathbb{R}^{d}_{\beta_n M^+}}} \int_{n|\bds| \leq   \frac{\theta_{0}}{1-\theta_{0}}x_{d}} \frac{|u(\bdx + {\bds}) - u(\bdx )|^{p} } {|x_{d}|^{\mu}}d\bds d\bdx,\quad\\
& = n^{p+d}(1 + \theta_{0})^{\mu}  \frac{1}{n^{\mu}}\frac{\theta_{0}^{\mu}}{(1-\theta_{0})^{\mu}} \int_{ {\mathbb{R}^{d}_{\beta_n M^+}}} \int_{|\bds| \leq  \frac{\theta_{0}}{1-\theta_{0}}\frac{x_{d}}{n}} \frac{|u(\bdx + {\bds}) - u(\bdx )|^{p} } {|{\theta_{n}} x_{d}|^{\mu} }d\bds d\bdx\\
& = n^{p-ps}(1 + \theta_{0})^{\mu}  \frac{\theta_{0}^{\mu}}{(1-\theta_{0})^{\mu}}\int_{ {\mathbb{R}^{d}_{\beta_n M^+}}} \int_{|\bds| \leq  \frac{\theta_{0}}{1-\theta_{0}}\frac{x_{d}}{n}} \frac{|u(\bdx + {\bds}) - u(\bdx )|^{p} } {|{\theta_{n}} x_{d}|^{\mu} }d\bds d\bdx\\
& = 
\theta_n^{ps-p} (1 + \theta_{0})^{\mu}  \left( {\frac{\theta_{0}}{1-\theta_{0}}}\right)^{d+p} \int_{ {\mathbb{R}^{d}_{\beta_n M^+}}} \int_{|\bds| \leq  \theta_n x_d} \frac{|u(\bdx + {\bds}) - u(\bdx )|^{p} } {|{\theta_{n}} x_{d}|^{\mu} }d\bds d\bdx\\
\end{split}
\]
 {Notice that for $\bx\in \mathbb{R}^{d}_{\beta_n M^+}$ and $|\bs|\leq \theta_n x_d$, we have
\[
[\bx+\bs]_i \leq \beta_n M + \beta_n \theta_n M = \left( 1- \frac{n^2}{(n-1)^2} \theta_0^2\right) M \leq M. 
\]} 
Therefore, we conclude that for each $n\in \mathbb{N}$,  there exists a constant $C = C(\theta_{0})$ such that
\begin{equation}\label{inter-ineq}
  I(\theta_{0}) \leq
  C(\theta_{0})(\theta_{n})^{ps-p}   \int_{\mathbb{R}^{d}_{M^+}} 
  \int_{B_{\theta_{n} x_{d}} (\bdx) \cap\mathbb{R}^{d}_{M^+} }
\frac{|u(\bdy) - u(\bdx )|^{p} } {|\theta_{n}x_{d}|^{\mu}}d\bdy d\bdx
\end{equation}
where  $ C(\theta_{0}) = (1 + \theta_{0})^{\mu} \left( {\frac{\theta_{0}}{1-\theta_{0}}}\right)^{d+p}$.   

Now given $\vartheta\in (0,\frac{\theta_{0}}{1-\theta_{0}}]$,  we can choose $n\in \mathbb{N}$ such that $  {\theta_{n+1}}  < \vartheta\leq \theta_{n}$.  This is possible since $\theta_{0}\in (0, 1)$. 
Then, apply the inequality \eqref{inter-ineq} first for $n+1$, and then later for $n$, we obtain that 
\[
\begin{split}
I(\theta_{0}) &\leq
  C(\theta_{0})( {\theta_{n+1}})^{ps-p}   \int_{\mathbb{R}^{d}_{M^+}} \int_{B_{ {\theta_{n+1}} x_{d}} (\bdx)\cap \mathbb{R}^{d}_{M^+}}
\frac{|u(\bdy) - u(\bdx )|^{p} } {| {\theta_{n+1}}x_{d}|^{\mu}}d\bdy d\bdx\\
&\leq C(\theta_{0})( {\theta_{n+1}})^{ps-p} \left(\frac{n+1}{n}\right)^{\mu}  \int_{\mathbb{R}^{d}_{M^+}} \int_{B_{\vartheta x_{d}} (\bdx)\cap \mathbb{R}^{d}_{M^+}}
\frac{|u(\bdy) - u(\bdx )|^{p} } {|\theta_{n}x_{d}|^{\mu}}d\bdy d\bdx\\
&\leq C(\theta_{0})(\theta_{n})^{ps-p} \left(\frac{n+1}{n}\right)^{d + p}  \int_{\mathbb{R}^{d}_{M^+}} \int_{B_{\vartheta x_{d}} (\bdx)\cap \mathbb{R}^{d}_{M^+}} \frac{|u(\bdy) - u(\bdx )|^{p} } {|\vartheta x_{d}|^{\mu}}d\bdy d\bdx\\
&\leq C(\theta_{0})2^{d + p}\vartheta^{ps-p}  
{|u|^{p}_{\mathfrak{W}^{s,p}_{\vartheta}(\mathbb{R}^{d}_{M^+})}}
\end{split}
\]
where in the last inequality we used the fact that $n\geq 1$ and $ps - p\leq 0$. That completes the proof of the lemma. 
\end{proof}

\section{Nonlocal Hardy-type inequalities}\label{N-Hardy}
We begin establishing a Hardy-type inequality for functions define on a line segment in one dimension. The result will also be  used later to prove the trace theorem in for  general domains. We start with the following basic estimates that follows from convexity of the map $t\mapsto t^{p}$, $1\leq p<\infty$. 
\begin{lem}
\label{inequalities} 
The following estimates hold. 
\begin{enumerate}
\item 
Suppose that $1\leq p < \infty.$ For any $\epsilon > 0$, there exists $C = C(\epsilon, p) > 1 $ such that for any $\xi\geq 0$, 
$
\xi^{p} - (1 + \epsilon) \leq C |\xi-1|^{p}.
$
\item For any $0\leq a < b \leq 1$, and any $1 <q < \infty$, we have 
$
\frac{b^{q} - a^{q}}{q(b-a)} < 1.
$
\end{enumerate}
\end{lem}

\begin{proof} Indeed, (2) follows from the intermediate value theorem, since there exists 
$c\in(a, b)$ such that
$$\frac{b^q-a^q}{q(b-a)}=\frac{1}{b-a}\int_a^b t^{q-1}dt =c^{q-1}<1.$$
Meanwhile, for any $\alpha\in (0,1)$,
$\alpha \xi=\alpha (\xi-1) + (1-\alpha)\frac{ \alpha}{1-\alpha}$,
so $\alpha^p \xi^p \leq \alpha (\xi-1)^p + (1-\alpha)^{1-p}\alpha^p$, that is, $ \xi^p \leq \alpha^{1-p} (\xi-1)^p + (1-\alpha)^{1-p}$. For $p=1$, this implies (1). For $p>1$ and any $\epsilon>0$, we can choose $\alpha=1-(1+\epsilon)^{1/(1-p)}$ and get the constant $C$ in (1) accordingly.
\end{proof}

\begin{prop}\label{1d-trace}
Let  $0\leq a<b\leq 1$, $1\leq p<\infty$,  and $s\in (0,1]$ such that $sp > 1$. Then there exists $C(s,p, a, b)>0$ such that for any $M> 0$ and $u\in C^{0,1}([0, M])$ with $u(0) = 0$ we have
\[
\int_{0}^{M} \frac{|u(x)|^{p}}{|x|^{sp}}dx \leq C \int_{0}^{M} \int_{ax}^{b x} \frac{|u(y) - u(x)|^{p}}{| x|^{sp + 1}}dy dx.
\]
\end{prop}
\begin{proof} The proof is similar to the $p=2$ case shown in \cite{TiDu2017}.
Given any $x, y\in (0, M)$ and $\epsilon > 0$, by Lemma \ref{inequalities}, there exists a constant $C(\epsilon, p) > 1$ such that 
\[
|u(x)|^{p} \leq C(\epsilon,p) |u(x) - u(y)|^{p}  + (1 + \epsilon)|u(y)|^{p}. 
\]
This is possible by taking $\xi = \frac{|u(x)|}{|u(y)|}$, when $u(y) \neq 0$  in Lemma \ref{inequalities}. 
Clearly the estimate is also true when $u(y) = 0$ since $C > 1$. 
Integrating the above inequality in $y$ in the interval $(ax, bx)$ and then integrating in $x$ from $0$ to $M$,  we get 
\[
\begin{split}
\int_{0}^{M} \frac{|u(x)|^{p}}{|x|^{sp}}dx &\leq \frac{C}{(b-a)} \int_{0}^{M} \int_{ax}^{bx} \frac{|u(x) - u(y)|^{p}}{|x|^{sp + 1}}dydx + \frac{1  +\epsilon} {b-a} \int_{0}^{M} \int_{ax}^{bx}\frac{|u(y)|^{p}}{|x|^{sp + 1}}dydx\\
& = I_{1} + I_{2}.
\end{split}
\]
$I_{1}$ is clearly in the right form, what remains is to estimate $I_{2}$. 
 {Since} $u\in C^{0,1}(\overline{\Om})$ and $u(0)=0$, we can use  Fubini's theorem and change the order of integration to get
\[
\begin{split}
I_{2}&= \frac{1+\ep}{b-a}\int_0^M \int_{ax}^{bx} \frac{| u(y)|^p}{|x|^{ps + 1}} dydx = \frac{1+\ep}{b-a}\int_0^{bM} \int_{y/b}^{{\min}\{y/a,M\}} \frac{| u(y)|^p}{|x|^{ps + 1}} dxdy \,.
\end{split} 
\]
Now, using the fact that $0\leq a<b\leq 1$, we  obtain that 
\[
\begin{split}
I_{2} &\leq \frac{1 + \epsilon}{b-a} \int_{0}^{M} \int_{y/b}^{y/a} \frac{|u(y)|^{p}}{|x|^{sp + 1}}dxdy \leq \frac{(1 + \epsilon)(b^{sp}-a^{sp})}{sp(b-a)} \int_{0}^{M}   \frac{|u(y)|^{p}}{|y|^{sp}}dy. 
\end{split}
\]
Again we use item (2) of Lemma \ref{inequalities} for $q = sp > 1$ to choose $\epsilon$ small that 
\[
 \frac{(1 + \epsilon)(b^{sp}-a^{sp})}{sp(b-a)}  =: c <1. 
\]
 {
Combining the above estimates we have
\[
(1-c) \int_{0}^{M} \frac{|u(x)|^{p}}{|x|^{sp}}dx \leq I_1,
\]
and the proposition is shown.
}

\end{proof}
 
Now that the one dimensional nonlocal Hardy-type inequality   {is} established, it is natural to expect similar result holds true for higher-space dimensions as well. 
In fact, one expects the existence of a constant  $C$ depending only on $s,p $ and $d$ such that 
 for any $M>0$
 and ${\bx} = (\bx', x_{d})$, 
\beq
 \int_{\mathbb{R}^{d}_{M^+} } \frac{|u(\bx)|^p}{|x_{d}|^{ps}} d\bx \leq C \int_{\mathbb{R}^{d}_{M^+} } \int_{\mathbb{R}^{d}_{M^+} \cap B_{x_{d}}(\bx)} \frac{|u(\by) - u({\bx })|^{p}}{|x_{d}| ^{{\mu} }} d\bdy d\bdx
 \,.
  \label{eq:hardyinstripe}
 \eeq
for any  $u\in C^{0,1}(\overline{\mathbb{R}^{d}_{M^+} })\cap \mathfrak{W}^{s,p}(\mathbb{R}^{d}_{M^+})$ and $u(\bx', 0)=0$ for $\bx'\in\R^{d-1}$. This is a special case of Theorem \ref{hardy-for-smooth} with $\vartheta=1$.

Let us begin the proof by making use of the one dimensional Hardy-type inequality 
in Proposition \ref{1d-trace}  {on} the function $u(\bx', \cdot)$ for each $\bx'\in \mathbb{R}^{d-1}$. Indeed, for any $0\leq a< b\leq 1 $ there exists a constant $C$ depending only on $a, b, p, s$ and $d$ such that 
\begin{equation}\begin{aligned}\label{H-I-inequality}
 \int_{\mathbb{R}^{d}_{M^+} } \frac{|u(\bx)|^p}{|x_d|^{ps}} d {\bx'} &= \int_{\R^{d-1}}\int_{0}^{M}  \frac{|u(\bx', x_{d})|^p}{|x_d|^{ps}} dx_d d\bx' \\
 &\leq C \int_{\R^{d-1}}\int_{0}^{M} \int_{a x_d}^{b x_d} \frac{|u(\bx', y_{d})-u(\bx', x_{d})|^p}{|x_d|^{1 + ps}} dy_ddx_d d\bx'\,,
\end{aligned}
\end{equation}
where we used the fact that, for each $\bx'\in \mathbb{R}^{d-1}$, the function $u(\bx',\cdot) \in C^{1}[0, M]$, $u(\bx',0) = 0$, and 
$\bdx = (\bx', x_{d}) \in \mathbb{R}^{d}_{M^{+}}$.
To complete the proof of \eqref{eq:hardyinstripe} we only need to appropriately estimate the integral on the right hand side of the above inequality.
To that end, we follow \cite{TiDu2017} and observe that the integral involves weighted variations of the function $u$ in the $d${th} variable.  We first present a definition on the nonlocal analog of norms of nonlocal directional derivatives, which is an extension to similar concepts first introduced in \cite{TiDu2017} for $p=2$.

\begin{defn}
\label{defn:directionalnorms}
 Suppose $0\leq a<b\leq 1$ and  $\kappa, \vartheta \in (0, 1]$ are given.  On the horizontal stripe  $  \mathbb{R}^{d}_{M^+}$, 
 we define in the following two directional nonlocal   seminorms $[\cdot]_n$ and $[\cdot]_t$,
 standing for normal and tangential directions respectively with reference to the boundary segment $\Gamma=\R^{d-1}\times \{0\}$,
 \begin{align}
 \label{def:dr1} [u]_{n}^p&=\int_{\mathbb{R}^{d}_{M^+} }\fint_{a x_d}^{ b x_d} \frac{|u(\bx',y_d)-u(\bx',x_d)|^p}{|\vartheta x_d|^{ps}} dy_ddx_d d\bx' \\
  \label{def:dr2} [u]_{t}^p&= \int_{\mathbb{R}^{d}_{M^+} } \fint_{B_{\kappa x_d}(\bx')}\frac{|u(\by', x_{d})-u(\bx', x_{d})|^p}{ |\vartheta x_d|^{ps}}  d\by'  dx_d d\bx'
 \end{align}
  {where $\fint$ denotes the average integral.}
 \end{defn}
 In the definition, the two seminorms depend on $a,b, \kappa$ and $\vartheta$. To eliminate the proliferation of messy notation, we suppress the dependence on the these constants. We now have the control on these seminorms.
 
 \begin{lem}\label{normal-tangent}
Let $\vartheta, s \in (0, 1]$,  $p\geq 1$, such that $sp > 1$. 
Then there exist constants $C$,  $a, b, $ and $\kappa$ with the property that $0\leq a<b\leq 1$, $0 < \kappa<1$ and $(a-1)^{2} + \kappa^{2} \leq \vartheta^{2}$ such that 
\begin{equation}\label{estimate-for-normal-tangential}
 [u]_{n} \leq C |u|_{\mathfrak{W}^{s,p}_{\vartheta}(\mathbb{R}^{d}_{M^+})} {,  \quad \text{ and }}
 \quad [u]_{t} \leq C |u|_{\mathfrak{W}^{s,p}_{\vartheta}(\mathbb{R}^{d}_{M^+})}
\end{equation}
for any $u\in \mathfrak{W}^{s,p}_{\theta}(\mathbb{R}^{d}_{M^+})$.
The constants depend on $\vartheta, p, s, $ and $d$. 
\end{lem}
\begin{proof}
It suffices to prove the lemma for the class $C^{0,1}_c(\overline{\mathbb{R}^{d}_{M^+}})$, and then take the  {completion} with respect to the $\mathfrak{W}^{s,p}_{\vartheta}(\mathbb{R}^{d}_{M^+})$ norm. 
To prove  the two inequalities in \eqref{estimate-for-normal-tangential} for any $ C^{0,1}_c(\overline{\mathbb{R}^{d}_{M^+}})$, it is sufficient to show the following two inequalities instead :
there exists constants $\tau_1$ and $\tau_{2}$ such that  for any $u$ measurable in $\mathbb{R}^{d}_{M^+} $, 
\begin{align}
\label{inequ1nt} [u]_n^p &\leq \tau_1 [u]_t^p + C\int_{\mathbb{R}^{d}_{M^+} } \int_{\mathbb{R}^{d}_{M^+} \cap B_{\vartheta x_{d}}(\bdx) } \frac{|u(\bdy) - u(\bdx)|^{p}}{|\vartheta x_{d}|^{{\mu} }} d\bdy d\bdx \\
\label{inequ2nt} [u]_t^p&\leq \tau_2  [u]_n^p+ C \int_{\mathbb{R}^{d}_{M^+} }\int_{\mathbb{R}^{d}_{M^+} \cap B_{\vartheta x_{d}} (\bdx)} \frac{|u(\bdy) - u(\bdx)|^{p}}{|\vartheta x_{d}|^{{\mu} }} d\bdy d\bdx
\end{align}
where the product $\tau_1\tau_2< 1$.  {Notice that for $u\in C^{0,1}_c(\overline{\mathbb{R}^{d}_{M^+}})$,
$[u]_{n}^{p}$, $[u]_{t}^{p}$ and  $|u|_{\mathfrak{W}^{s,p}_{\vartheta,M}(\mathbb{R}^{d}_{M^+})}$ are all finite.} 
We now focus on establishing \eqref{inequ1nt} and \eqref{inequ2nt}.  For any $x_{d}, y_{d} \in (0, M)$, and $\bx', \by'\in\mathbb{R}^{d-1}$,  we may write 
\[
u(\bdx', y_{d}) - u(\bdx', x_{d}) = u(\bdx', y_{d}) - u(\bdy', y_{d})   + u(\bdy',y_{d}) -u(\bdx', x_{d}). 
\]
 \begin{figure}[htbp] 
\centering
  \begin{tikzpicture}[scale=0.9]    
     \draw[line width=2pt] (-2.8,0)--(2.8,0);       
      \draw[line width=2pt] (0,0)--(0, 1.05);
            \draw[line width=2pt] (0,2.75)--(0,3.2);     
            
         \draw[line width=2pt,purple,dotted] (0,2.7)--(0,1.1);
          \draw [decorate,thick, opacity=0.6,decoration={brace,amplitude=10pt, mirror},yshift=4pt](0,0.9) -- (0,2.6) ;

          \draw [dotted, line width=2pt](0,3.2) circle(3.2);
      \draw [line width=2pt,red,dashed](-1.6,3.2)--(1.6,3.2);   
             \draw[line width=2pt,blue,dotted] (-1.6, 2.3)--(1.6,2.3);   
    \draw (0, 3.2) circle(3pt) [fill] ;
       \draw (0, 2.3) circle(3pt) [fill] ;
       \draw (1.2, 2.3) circle(3pt) [fill] ; 
              \draw (1.2,3.2) circle(3pt) [fill] ;    
         
        \node at (-0.3,3.6) {$(\bx', x_{d})$};         
       \node at (-0.8,2.6) {$(\bx', y_{d})$};      
        \node at (1.4, 2.6) {$(\by', y_{d})$};
                \node at (1.4,3.6) {$(\by', x_{d})$};
         \node at (-3.2, 3.4) { {$B_{\kappa y_d}\hspace{-2pt}(\bx')\times \{x_{d}\}$}};  
                  \node at (-3.2, 2.4) { {$B_{\kappa y_d}\hspace{-2pt}(\bx')\times \{y_{d}\}$}};  
                 \node at (1.2, 1.4) {$(ax_d,bx_d)$};  

        \node at (4.6,0.1) {$\R^{d-1}\times \{0\}$};
    \end{tikzpicture}
      \caption{Depiction of geometry  used in the proof of Lemma \ref{normal-tangent}.}\label{fig1}
      \end{figure}
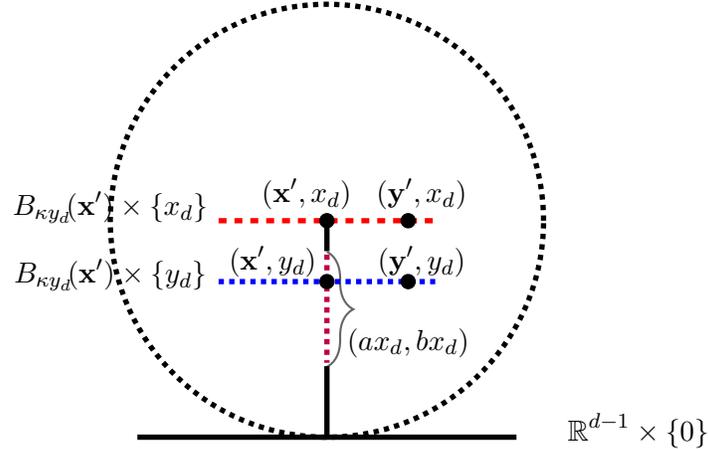

Now given any any $\epsilon > 0$, we may apply Lemma \ref{inequalities} to obtain $C_{\epsilon} > 1$ such that 
\[
|u(\bdx', y_{d}) - u(\bdx', x_{d})|^{p} \leq (1 + \epsilon)|u(\bdx', y_{d}) - u(\bdy', y_{d})|^{p}   + C_{\epsilon} |u(\bdy',y_{d}) -u(\bdx', x_{d})|^{p}.
\]
We fix constants $0 < a <b\leq 1$ and  $\kappa \in (0, 1)$ to be determined shortly. 
Integrating on both sides of the above inequality with respect to $\bdy'$ over the ball $B_{\kappa  y_{d}}(\bdx')$, we obtain that 
\[
\begin{split}
&|u(\bdx', y_{d}) - u(\bdx', x_{d})|^{p} \\ \leq  &(1 + \epsilon)\fint_{B_{ \kappa y_{d}}(\bdx')}|u(\bdx', y_{d}) - u(\bdy', y_{d}) |^{p} d\bdy'  + C_{\epsilon} \fint_{B_{\kappa y_{d}}(\bdx')} |u(\bdy',y_{d}) -u(\bdx', x_{d})|^{p}d\bdy'.
\end{split}
\]
We then integrate in the $y_{d}$ variable over the interval $(ax_{d},bx_{d})$ both sides of the above to obtain  
  \[
\begin{split} \fint_{ a x_{d}}^{bx_{d}} {|u(\bdx', y_{d}) - u(\bdx', x_{d})|^{p}}dy_{d}
&\leq (1 + \epsilon)\fint_{ a x_{d}}^{ bx_{d}}\fint_{B_{ \kappa y_{d}}(\bdx')}{|u(\bdx', y_{d}) - u(\bdy', y_{d}) |^{p}} d\bdy'dy_{d}\\
&+ C_{\epsilon}\fint_{ a x_{d}}^{ bx_{d}}\fint_{B_{ \kappa y_{d}}(\bdx')}|u(\bdy', y_{d}) - u(\bdx', x_{d}) |^{p} d\bdy'dy_{d}\\
\end{split}
\]
Dividing both sides of the inequality by $|\vartheta x_{d}|^{ps}$, and integrating over $\mathbb{R}^{d}_{M^+} $, we obtain that 
\[
\int_{\R^{d-1}} \int_{0}^{M}  \fint_{ a x_{d}}^{bx_{d}} \frac{|u(\bdx', y_{d}) - u(\bdx', x_{d})|^{p}}{|\vartheta x_{d}|^{ps}}dy_{d} dx_{d} d\bdx' = J_{1} + J_{2}
\]
where 
\[
J_{1} = (1 + \epsilon)\int_{\R^{d-1}} \int_{0}^{M}  \fint_{ a x_{d}}^{ bx_{d}}\fint_{B_{ \kappa y_{d}}(\bdx')}\frac{|u(\bdx', y_{d}) - u(\bdy', y_{d}) |^{p}}{|\vartheta x_{d}| ^{ps}} d\bdy'dy_{d}d x_{d} d\bdx'
\]
and 
\[
J_{2} = C_{\epsilon} \int_{\R^{d-1}} \int_{0}^{M}  \fint_{ a x_{d}}^{ bx_{d}}\fint_{B_{ \kappa y_{d}}(\bdx')}\frac{|u(\bdy', y_{d}) - u(\bdx', x_{d}) |^{p}}{|\vartheta x_{d}| ^{ps}} d\bdy'dy_{d}d x_{d} d\bdx'. 
\]
The integral $J_{1}$ can be estimated as follows using Fubini's theorem and a change of variables: 
\begin{align*}
J_{1}  &= \frac{(1 + \epsilon)\vartheta}{b-a} \int_{\R^{d-1}}\int_{0}^{M}  \int_{ a x_{d}}^{ bx_{d}}\fint_{B_{ \kappa y_{d}}(\bdx')}\frac{|u(\bdx', y_{d}) - u(\bdy', y_{d}) |^{p}}{|\vartheta x_{d}| ^{ps + 1}} d\bdy'dy_{d}d x_{d} d\bdx'\\
& \leq \frac{ (1 + \epsilon)\vartheta}{b-a}\int_{\mathbb{R}^{d-1}} \int_{0}^{M}\left( \int_{{y_{d}\over b}}^{{y_{d}\over a}}\fint_{B_{\kappa y_{d}}(\bdx')}\frac{|u(\bdx',y_{d}) - u(\bdy', y_{d})|^{p}}{|\vartheta x_{d}|^{ps+1}}  d\bdy'dx_{d}\right)dy_{d} d\bdx'
\end{align*}
In the last integral we iterate the integrals first in $x_{d}$  to obtain 
\[
\begin{split}
J_{1} & \leq \frac{ (1 + \epsilon)(b^{sp}-a^{sp})}{sp(b-a)} \int_{\mathbb{R}^{d-1}} \int_{0}^{M} \fint_{B_{\kappa y_{d}}(\bdx')}\frac{|u(\bdx',y_{d}) - u(\bdy', y_{d})|^{p}}{|\vartheta y_{d}|^{ps}} d\bdy'dy_{d} d\bdx' \\
&=\frac{ (1 + \epsilon)(b^{sp}-a^{sp})}{sp(b-a)}  {[u]_{t}^{p}.}
\end{split}
\] 
Next we bound $J_{2}$. Clearly,  since $a x_{d} < y_{d} < b x_{d}$, we have that measure of the ball $B_{\kappa y_{d}} (\bx')$ in $\mathbb{R}^{d-1}$ can be estimated in terms of $a$ and $x_{d}$ as 
\[
\begin{split}
 &J_{2}=  \frac{C_{\epsilon}\vartheta }{b-a} \int_{\R^{d-1}} \int_{0}^{M}  \int_{ a x_{d}}^{ bx_{d}}\fint_{B_{ \kappa y_{d}}(\bdx')}\frac{|u(\bdy', y_{d}) - u(\bdx', x_{d}) |^{p}}{|\vartheta x_{d}| ^{ps +1}} d\bdy'dy_{d}d x_{d} d\bdx'\\
 &\leq    \frac{C_{\epsilon}  \vartheta^{d}}{(b-a)|\kappa a|^{d-1}}\int_{\R^{d-1}} \int_{0}^{M}  \int_{ a x_{d}}^{ bx_{d}}\int_{B_{ \kappa y_{d}}(\bdx')}\frac{|u(\bdy', y_{d}) - u(\bdx', x_{d}) |^{p}}{|\vartheta x_{d}| ^{\mu}} d\bdy'dy_{d}d x_{d} d\bdx'. 
 \end{split}
\]
Now for any $\bdy=(\bdy', y_{d}) \in  B_{ \kappa y_{d}}(\bdx')\times (a x_{d} , bx_{d})$ and $\bdx = (\bdx', x_{d}) \in \mathbb{R}^{d}_{M^+} $, we have that $\by \in \mathbb{R}^{d}_{M^+} $ and 
\[
|\bdy - \bdx|^{2}  = (y_{d}-x_{d})^{2} + |\bdy'-\bdx'|^{2} \leq (a-1)^{2}x_{d}^{2} + \kappa^{2}b^{2} x_{d}^{2} \leq ((a-1)^{2} + \kappa^{2}) x_{d}^{2} < \vartheta^{2} x_{d}^{2}
\] 
where we have chosen $a$ and $\kappa$ in such a way that $(a-1)^{2} + \kappa^{2} < \vartheta^{2}$, which is always possible to do. It follows then that 
\[
J_{2} \leq  \frac{C_{\epsilon }\vartheta^{d}}{(b-a)| \kappa a|^{d-1}}\int_{  \mathbb{R}^{d}_{M^+} }\int_{  \mathbb{R}^{d}_{M^+} \cap B_{\vartheta x_{d}} (\bdx)} \frac{|u(\bdy) -u(\bdx)|^{p}}{|\vartheta x_{d}|^{{\mu} }}d\bdy d\bdx. 
\]
Putting together, we have  just demonstrated that 
\[
[u]_{n}^{p}\leq \tau_{1} [u]_{t}^{p} + C\int_{  \mathbb{R}^{d}_{M^+} }\int_{  \mathbb{R}^{d}_{M^+}  \cap B_{\vartheta x_{d}}(\bx)} \frac{|u(\bdy) -u(\bdx)|^{p}}{|\vartheta x_{d}|^{{\mu} }}d\bdy d\bdx ,\quad \text {with $\tau_{1} = \frac{ (1 + \epsilon)(b^{sp}-a^{sp})}{sp(b-a)}.$ }
\]
We next estimate $[u]_{t}^{p}$. Following the same procedure as above given any $\epsilon > 0$ we can find $C_{\epsilon}$ such that 
\[
\begin{split}
[u]_{t}^{p}&\leq (1 + \epsilon)  \int_{\mathbb{R}^{d-1}}\int_{0}^{M}\fint_{B_{\kappa x_{d}}(\bdx')}\fint_{ax_{d}}^{bx_{d}} \frac{|u(\bdy',x_{d}) - u(\bdy',y_{d})|^{p}}{|\vartheta x_{d}|^{ps}} dy_{d} d\bdy' dx_{d} d\bdx' \\
&+  C_{\epsilon} \int_{\mathbb{R}^{d-1}}\int_{0}^{M}\fint_{B_{\kappa x_{d}}(\bdx')} \fint_{ax_{d}}^{bx_{d}}\frac{|u(\bdy',y_{d}) - u(\bdx', x_{d})|^{p}}{|\vartheta x_{d}|^{ps}} dy_{d} d\bdy' dx_{d} d\bdx'\\
& = J_{3} + J_{4} {.}
\end{split}
\]
Using Fubini's theorem and the observation that  $\bdy' \in B_{\kappa x_{d}}(\bdx')$ if and only if $\bdx' \in B_{\kappa x_{d}}(\bdy')$, $J_{3}$  can be rewritten as
{
 \[
\begin{split}
J_{3}&= (1 + \epsilon)  \int_{\mathbb{R}^{d-1}}\int_{0}^{M}\fint_{B_{\kappa x_{d}}(\bdx')}\fint_{ax_{d}}^{bx_{d}} \frac{|u(\bdy',x_{d}) - u(\bdy',y_{d})|^{p}}{|\vartheta x_{d}|^{ps}} dy_{d} d\bdy' dx_{d} d\bdx'\\
  &= (1 + \epsilon)  \int_{0}^{M} \int_{\mathbb{R}^{d-1}}\fint_{B_{\kappa x_{d}}(\bdy')}\fint_{ax_{d}}^{bx_{d}}\frac{|u(\bdy',x_{d}) - u(\bdy',y_{d})|^{p}}{|\vartheta x_{d}|^{ps}} dy_{d} d\bdx'd\bdy' dx_{d} \\
  &= (1 + \epsilon)  \int_{0}^{M}\int_{\mathbb{R}^{d-1}}\fint_{ax_{d}}^{bx_{d}}\frac{|u(\bdy',x_{d}) - u(\bdy',y_{d})|^{p}}{|\vartheta x_{d}|^{ps} } dy_{d}  d\bdy'dx_{d}\\
  & = (1 + \epsilon)   {[u]_{n}^{p}}. 
\end{split}
 \]
 }
The integral $J_{4}$ can be controlled the same as $J_{2}$ before as follows 
\[
\begin{split}
J_{4}
&\leq \frac{C_{\epsilon} \vartheta}{b-a}\int_{\mathbb{R}^{d-1}}\int_{0}^{M}\fint_{B_{\kappa x_{d}}(\bdx')} \int_{ax_{d}}^{bx_{d}}\frac{|u(\bdy',y_{d}) - u(\bdx',x_{d})|^{p}}{|\vartheta x_{d}|^{ps+1}}dy_{d}d\bdy' dx_{d} d\bdx'\\
& = \frac{C_{\epsilon} \vartheta^{d}}{b-a} \int_{  \mathbb{R}^{d}_{M^+} }\int_{  \mathbb{R}^{d}_{M^+} \cap B_{\vartheta x_{d}} (\bx)} \frac{|u(\bdy) -u(\bdx)|^{p}}{|\vartheta x_{d}|^{{\mu} }}d\bdy d\bdx,
\end{split}
\]
 provided that $ (a-1)^{2} + \kappa^{2} < \vartheta^{2}$.  We thus  {have}
 \begin{equation}
[u]_{t}^{p} \leq \tau_{2} [u]_{n}^{p} + C  \int_{  \mathbb{R}^{d}_{M^+} }\int_{  \mathbb{R}^{d}_{M^+} \cap B_{\vartheta x_{d}} (\bx)} \frac{|u(\bdy) -u(\bdx)|^{p}}{|\vartheta x_{d}|^{{\mu} }}d\bdy d\bdx ,\quad \text{with   $\tau_{2} = (1 + \epsilon)$}
\end{equation}
Now, since  $sp > 1$, by  {Lemma \ref{inequalities}}, we have $ \frac{ (b^{sp}-a^{sp})}{sp(b-a)} < 1$.  {Therefore} we can find $\epsilon $ small such that 
\[
\tau_{1}\tau_{2} =  \frac{ (1 + \epsilon)^{2}(b^{sp}-a^{sp})}{sp(b-a) }< 1. 
\]
In the event that both $[u]_{t}^{p} $ and $[u]_{n}^{p} $ are finite, we obtain \eqref{estimate-for-normal-tangential}. 
\end{proof}
We are now in a position to prove the Hardy-type inequality for functions defined on the stripe $\mathbb{R}^d_{M^+}$ that uses the seminorm  $|\cdot|_{\mathfrak{W}^{s,p}_{\vartheta}(\mathbb{R}^{d}_{M^+})}$. A special case of this theorem is stated in  {Remark \ref{rem:hardy-for-smooth-special}} for the half space. 
\begin{thm}[Hardy-type inequality]\label{hardy-for-smooth}
Suppose that $p\geq 1$, $s\in (0,1]$ such that $sp > 1$. Let also  $\vartheta \in (0, 1]$. There exists a constant $C = C(d, p, s,\vartheta)$ such that for any $M>0$ if  $u\in C^{0,1}(\overline{\mathbb{R}^{d}_{M^{+}}})\cap \mathfrak{W}^{s,p}_{\vartheta}(\mathbb{R}^{d}_{M^+})$ and $u|_{x_{d} = 0} = 0$, then
\[
\left(\int_{\mathbb{R}^{d}_{M^+}}\frac{|u(\bdx)|^{p}}{|x_{d}|^{ps}} d\bdx \right)^{1/p}\leq C |u|_{\mathfrak{W}^{s,p}_{\vartheta}(\mathbb{R}^{d}_{M^+})}.  
\]
\end{thm}


\begin{proof}
Recall from \eqref{H-I-inequality} that, given any $0\leq a<b\leq 1$, there exists a positive constant $C$,
 {depending on $\vartheta,p,s$,and $d$} such that 
\[
\begin{split}
\int_{\mathbb{R}^{d}_{M^+}}\frac{|u(\bdx)|^{p}}{|x_{d}|^{ps}} d\bdx &\leq C \int_{\R^{d-1}} \int_{0}^{M} \fint_{a x_{d}}^{bx_{d}} \frac{u(\bdx', y_{d}) - u(\bdx', x_{d})}{|x_{d}|^{ps}} dy_{d} d x_{d} d{\bx'}\\
&= C \vartheta^{ps}  {[u]_{n}^{p}}
\end{split}
\]
where the equality is by definition of the directional derivative.  
We can now apply the Lemma \ref{normal-tangent}
 {and}
choose $a$ and $b$ such that the latter can be bounded by the seminorm $|u|_{\mathfrak{W}^{s, p}_{\vartheta} (\mathbb{R}^{d}_{M^{+}})}$to prove the theorem. 
\end{proof}

The dependence of the constant $C$ in Theorem \ref{hardy-for-smooth} on the parameter $\vartheta$ can be made explicit by using the  {$L^p$-norm} in the right hand side.  
\begin{coro}\label{cor-hardyonRM}
Suppose that $\vartheta \in (0, 1]$, $p\geq 1$, $s\in (0,1)$ such that $sp > 1$.   
Then there exists a constant $C = C(d, p, s)$ but independent of $\vartheta$ such that for any $u\in C^{0,1}(\overline{\mathbb{R}^{d}_{M^+}})\cap \mathfrak{W}^{s,p}_{\vartheta}(\mathbb{R}^{d}_{M^+})$ and $u|_{x_{d}=0} = 0$, 
\[
\int_{\mathbb{R}^{d}_{M^+}}\frac{|u(\bdx)|^{p}}{|x_{d}|^{ps}} d\bdx \leq C \left(\, {\vartheta^{ps- p} 
|u|^{p}_{\mathfrak{W}^{s,p}_{\vartheta}(\mathbb{R}^{d}_{M^+})}
+  M^{-ps} \|u\|^{p}_{L^p(\mathbb{R}^{d}_{M^+})}}
\right). 
\]
\end{coro}

\begin{proof}
We apply the Theorem \ref{hardy-for-smooth} corresponding to $\vartheta=1/2$ to get a constant $C_{1}$ such that 
\[
\int_{\mathbb{R}^{d}_{M^+}}\frac{|u(\bdx)|^{p}}{|x_{d}|^{ps}} d\bdx \leq C_1 \int_{\mathbb{R}^{d}_{M^+}} \int_{\mathbb{R}^{d}_{M^+}\cap B_{\frac{1}{2} x_{d}} (\bx)}\frac{|u(\bdy) - u(\bdx)|^{p}}{|\frac{1}{2} x_{d}|^{d + ps}}d\bdy d\bdx,
\]
for any $u\in C^{0, 1}(\overline{\mathbb{R}^{d}_{M^+}})\cap \mathfrak{W}^{s,p}_{\vartheta}(\mathbb{R}^{d}_{M^+})$ and $u|_{x_{d}=0} = 0$. 
Next apply Lemma \ref{lemma-half-to-vartheta} corresponding to $\theta_{0}  = 1/2$ to obtain a constant $C$ independent of $\vartheta$ such that  \[
\begin{split}
\int_{\mathbb{R}^{d}_{M^+}}&\frac{|u(\bdx)|^{p}}{|x_{d}|^{ps}} d\bdx \leq C_1 \int_{\mathbb{R}^{d}_{M^+}} \int_{B_{\frac{1}{2} x_{d} }(\bdx) \cap \mathbb{R}^{d}_{M^+} }\frac{|u(\bdy) - u(\bdx)|^{p}}{|\frac{1}{2} x_{d}|^{d + ps}}d\bdy d\bdx\\
&\leq C \vartheta^{ps- p} \int_{\mathbb{R}^{d}_{M^+}} \int_{B_{\vartheta x_{d}}(\bdx) \cap \mathbb{R}^{d}_{M^+} } \frac{|u(\bdy) - u(\bdx)|^{p}}{|\vartheta x_{d}|^{{\mu}}}d\bdy d\bdx + C 2^{p-1}M^{-ps} \int_{\mathbb{R}^{d}_{M^+}}|u(\bdx)|^{p}d\bdx
\end{split}
\] for any $\vartheta \in (0, 1]$, and any $u\in C^{0,1}(\overline{\mathbb{R}^{d}_{M^+}})\cap \mathfrak{W}^{s,p}_{\vartheta}(\mathbb{R}^{d}_{M^+})$ and $u|_{x_{d}=0} = 0$. 
\end{proof}
\begin{remark}
\label{rem:hardy-for-smooth-special}
A simple consequence  Corollary \ref{cor-hardyonRM} is that by letting $M\to \infty$, we get a constant $C=C(d, p,s)$ such that for any $\vartheta\in (0,1]$
\[
\int_{\mathbb{R}^{d}_{+}}\frac{|u(\bdx)|^{p}}{|x_{d}|^{ps}} d\bdx \leq C \,\vartheta^{ps- p} \int_{\mathbb{R}^{d}_{+}} \int_{B_{\vartheta x_{d}}(\bx) }\frac{|u(\bdy) - u(\bdx)|^{p}}{|\vartheta x_{d}|^{{\mu}}}d\bdy d\bdx
\]
 for any $u\in  {C^{0,1}_c(\mathbb{R}^{d}_{+})}$,  {and thus for any $u\in \mathfrak{\mathring{W}}^{s,p}_{\vartheta}(\mathbb{R}^{d}_{+})$ by completion.}
\end{remark}

\section{Trace theorems on stripe domains} 

\label{trace-strip}
In this section we prove trace theorems for functions defined on stripes. We begin estimating norm of the trace of $C^{1}$ functions in terms of  the $\Nw$-norm on $\mathbb{R}^{d}_{M^+}$. More precisely we have the following. 
\begin{thm}\label{Trace-for-half-space}
Let $p\geq 1$, $s\in (0,1]$ such that $sp > 1$. 
There exists a constant $C = C(s, p, d)$ such that for any $u\in C^{0, 1}(\overline{\mathbb{R}^{d}_{M^+}})\cap \mathfrak{W}^{s,p}_{\vartheta}(\mathbb{R}^{d}_{M^+})$
and for any $\vartheta \in (0, 1]$ we have 
\[
\|u\|^{p}_{L^p(\Gamma)}
\leq C\left( M^{-1} \| u\|^{p}_{L^{p}(\mathbb{R}^{d}_{M^+})} + \vartheta^{ps-p}  M^{ps-1}
{
|u|^{p}_{\mathfrak{W}^{s,p}_{\vartheta}(\mathbb{R}^{d}_{M^+})}
}
\right).
\]
Moreover, 
\[
|u|^{p}_{W^{s-1/p, p}(\Gamma)} \leq  C \left(M^{-ps} \|u\|^{p}_{L^{p}(\mathbb{R}^{d}_{M^+})} + \vartheta^{ps-p}
{
|u|^{p}_{\mathfrak{W}^{s,p}_{\vartheta}(\mathbb{R}^{d}_{M^+})}
}
 \right).
\]
where  $|\cdot|_{W^{s-1/p, p}(\Gamma)}$ is the usual fractional Sobolev norm on the hypersurface $\Gamma = \mathbb{R}^{d-1}\times \{0\}$. 
\end{thm}
\begin{remark}
\label{rem:trace_half_space}
It will be clear in the proof that in the event that $\Omega = \mathbb{R}^{d}_{+}$, the inequalities in the theorem can be combined to yield the inequality that   
\[
\|u\|^{p}_{W^{s-1/p, p}(\Gamma)} \leq  C(p, d, s) \left(\|u\|^{p}_{L^{p}(\mathbb{R}^{d}_{+})
} + \vartheta^{ps-p} |u|^{p}_{\mathfrak{W}^{s, p}_{\vartheta}(\mathbb{R}^d_{M^+})}
  \right).
\]
\end{remark}
\begin{proof}
Since $u\in C^{0,1}(\overline{\mathbb{R}_{M^+}^{d}})$, pointwise evaluation on the boundary is well defined. For any $x_{d} >0$, we have that  
\[
|u(\bdx',0)|^{p} \leq 2^{p-1}( |u(\bdx', x_{d})|^{p}  +  |u(\bdx',x_{d}) - u(\bdx',0)|^{p} ,  \text{for all $\bdx'\in \mathbb{R}^{d-1}$}. 
\]
Taking the average  in $x_{d}$ over the interval $[0, M]$ first and then integrating in $\bdx'$ over $\mathbb{R}^{d-1}$, we obtain that
\begin{equation}\label{Tr-lp1}
\begin{split}
\int_{\mathbb{R}^{d-1}}|u(\bdx',0)|^{p}d\bdx' &\leq 2^{p-1}\left( \frac{1}{M}\int_{\mathbb{R}^{d-1}}\int_{0}^{M}|u(\bdx', x_{d})|^{p}dx_d d\bdx'  \right.\\
&+ \left.M^{ps-1} \int_{\mathbb{R}^{d-1}}\int_{0}^{M}
\frac{|u(\bdx',x_{d}) - u(\bdx',0)|^{p}}{|x_{d}|^{ps}}d x_{d} d\bdx'\right). 
\end{split}
\end{equation} 
We note that for each $\bdx' \in \mathbb{R}^{d-1}$ the function $x_{d} \mapsto u(\bdx',x_{d}) - u(\bdx',0)$ is in $C^{0, 1}[0, M]$ and vanishes on the hyperplane  $x_{d}= 0$. We may then apply the one space dimension Hardy's inequality  Proposition \ref{1d-trace} to obtain that given any $a,b$ such that $0\leq a < b \leq 1$,  there exists a corresponding constant $C > 0$ such that 
\[
\begin{split}
\int_{0}^{M}\frac{|u(\bdx',x_{d}) - u(\bdx',0)|^{p}}{|x_{d}|^{ps}}d x_{d}  &\leq C \int_{0}^{M} \int_{ax_{d}}^{bx_{d}} \frac{|u(\bdx',y_{d}) - u(\bdx',x_{d})|^{p}}{|x_{d}|^{ps + 1}}dy_{d} dx_{d}\\
& \leq C(b-a) \int_{0}^{M} \fint_{ax_{d}}^{bx_{d}}  \frac{|u(\bdx',y_{d}) - u(\bdx',x_{d})|^{p}}{|x_{d}|^{ps }}dy_{d} dx_{d}.\\
\end{split}
\]
We integrate the $\bdx'$ variables on $\mathbb{R}^{d-1}$ and rewrite it to obtain  
\[
\begin{split}
\int_{\R^{d-1}}\int_{0}^{M}&\frac{|u(\bdx',x_{d}) - u(\bdx',0)|^{p}}{|x_{d}|^{ps}}d x_{d} d\bdx'\\
&\leq \frac{C(b-a)}{2^{ps}} \int_{\mathbb{R}^{d-1}}\int_{0}^{M} \fint_{ax_{d}}^{bx_{d}} \frac{|u(\bdx',y_{d}) - u(\bdx', x_{d})|^{p}}{|\frac{1}{2}x_{d}|^{ps}}dy_{d} dx_{d} d\bdx' .
\end{split}
\]
The integral in the right hand side is precisely $  {[u]_{n}^{p}}$ corresponding to $\vartheta = 1/2$. 
 We now apply Lemma \ref{normal-tangent} to conclude that there exists constants $a, b,$ and $C$ such that 
 \[
 \int_{\R^{d-1}}\int_{0}^{M}\frac{|u(\bdx',x_{d}) - u(\bdx',0)|^{p}}{|x_{d}|^{ps}}d x_{d} d\bdx'
 \leq C \int_{\mathbb{R}^{d}_{M^+}} \int_{B_{\frac{1}{2} x_{\delta} }(\bdx)\cap \mathbb{R}^{d}_{M^+}  } \frac{|u(\bdy) - u(\bdx)|^{p}}{|\frac{1}{2}x_{d}|^{d + ps}} d\bdy d\bdx. 
 \]
Lemma \ref{lemma-half-to-vartheta} in turn guarantees the existence of a constant $C$ independent of $\vartheta$ such that 
\[\begin{split}
\int_{\R^{d-1}}\int_{0}^{M}&\frac{|u(\bdx', x_{d}) - u(\bdx',0)|^{p}}{|x_{d}|^{ps}}d x_{d} d\bdx' \\
&\leq C \left(\vartheta^{ps-p} \int_{\mathbb{R}^{d}_{M^+}} \int_{ B_{\vartheta x_{d}}(\bx)\cap  \mathbb{R}^{d}_{M^+}} \frac{|u(\bdy) - u(\bdx)|^{p}}{|\vartheta x_{d}|^{{\mu} }} d\bdy d\bdx +  M^{-ps} \|u\|_{L^{p}(\mathbb{R}^{d}_{M^+})}^{p}\right), 
\end{split}
\]
for all $\vartheta \in (0, 1].$
Together with the inequality \eqref{Tr-lp1}, we obtain that 
\[
\int_{\mathbb{R}^{d-1}}|u(0, \bdx')|^{p}d\bdx' \leq C\left( M^{-1} \| u\|^{p}_{L^{p}(\mathbb{R}^{d}_{M^+})} + \vartheta^{ps-p}  M^{ps-1} 
{
|u|^{p}_{\mathfrak{W}^{s,p}_{\vartheta, M}(\mathbb{R}^{d}_{M^+})}
}
\right),
\]
which proves the first inequality of the theorem. 
Let us prove the second inequality.
 By definition the seminorm $|u|_{W^{s-1/p, p}(\Gamma)} $ is given by 
 \[
 |u|^{p}_{W^{s-\frac{1}{p}, p}(\Gamma)}  = \int_{\mathbb{R}^{d-1}}\int_{\mathbb{R}^{d-1}} \frac{|u(\bdx',0) - u(\bdy',0)|^{p}}{|\bdx'-\bdy'|^{d + ps-2}}d\bdx'd\bdx. 
 \]
 We divide the domain of integration as follows 
\[\begin{aligned}
   |u|^{p}_{W^{s-\frac{1}{p}, p}(\Gamma)}&=  \int_{\mathbb{R}^{d-1}}\int_{B_{\frac{M}{2}}(\bdx')} \frac{|u(\bdx', 0) - u(\bdy',0)|^{p}}{|\bdx'-\bdy'|^{d + ps-2}}d\bdy'd\bdx' \\
&\qquad  +  \int_{\mathbb{R}^{d-1}}\int_{B_{\frac{M}{2}}(\bdx')^{C}} \frac{|u(\bdx',0) - u(\bdy',0)|^{p}}{|\bdx'-\bdy'|^{d + ps-2}}d\bdy'd\bdx'\\
  &= I_{1} + I_{2},
 \end{aligned}\]
 where $B_{\frac{M}{2}}(\bdx')^{C}$  represents the complement of $B_{\frac{M}{2}}(\bdx')$.
Estimating $I_{2}$ is relatively simple since 
 \[\begin{split}
   I_{2}&\leq  
   \int_{\mathbb{R}^{d-1}}\int_{B_{\frac{M}{2}}(\bdx')^{C}}\frac{ |u(\bdx',0) - u(\bdy',0)|^{p}}{|\bdx'-\bdy'|^{d + ps-2}}d\bdy'd\bdx' \\
  & \leq \int_{\{|\bdy'| \geq \frac{M}{2}\} }\int_{\mathbb{R}^{d-1}} \frac{|u(\bdx',0) - u( \bdx' + \bdy',0)|^{p}}{|\bdy'|^{d + ps-2}}d\bdx'd\bdy' \\
   & \leq 2^{p-1} \|u(\cdot,0 )\|_{L^{p}}^{p}  \int_{\{|\bdy'| \geq \frac{M}{2}\} } \frac{1}{|\bdy'|^{d + ps-2}} \leq C \|u(\cdot,0 )\|_{L^{p}(\mathbb{R}^{d-1})}^{p} \\
   &\leq C(p,s,d) M^{-ps + 1}  \|u(0,\cdot )\|_{L^{p}}^{p}
 \end{split}
 \]
 where we have used the assumption that $ps > 1$ that will imply the finiteness of the integral.  
 
 Let us estimate  the first term $I_{1}$.
      
      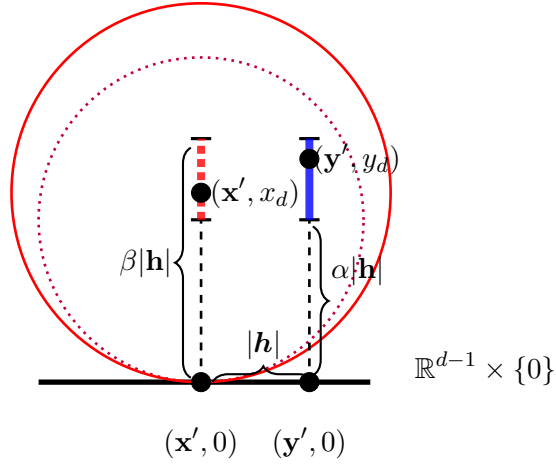
\begin{figure}[htbp] 
\centering
   \begin{tikzpicture}[scale=.9]
     \draw[thick, line width=2pt] (-2.4,0)--(2.5,0);
             \draw[dotted, line width=1pt, purple](0,2.4) circle(2.4);
        \draw[line width=1pt, red](0,2.8) circle(2.8);
        
      \node at (0,-0.9) {$(\bx',0)$};
       \node at (1.6,-0.9) {$(\by',0)$};
       
       \draw [decorate,thick, decoration={brace,amplitude=6pt},xshift=2pt]
(0.1,0) -- (1.6, 0) node [black,midway,yshift=0.5cm] {$|{\bm h}|$};

       \draw [decorate,thick, decoration={brace,amplitude=6pt},yshift=-1pt]
(-0.15,0.1)--(-0.15,3.5)  node [black,midway,xshift=-0.6cm] {$\beta|{\bf h}|$};

       \draw [decorate,thick, decoration={brace,amplitude=6pt},xshift=-0.1cm]
(1.75,2.3)--(1.75,0.1)node [black,midway,yshift=0.4cm,xshift=0.6cm]{$\alpha|{\bf h}|$};

       \draw[thick, line width=1pt, dashed] (1.6,0)--(1.6,2.4);  
       \draw[thick,line width=1pt, dashed] (0,0)--(0,2.4);  
       \draw (0,0) circle(4pt) [fill] ;
      \draw (1.6,0) circle(4pt) [fill] ;
            
       \draw[very thick] (-0.15,2.4)--(0.15,2.4);
       \draw[very thick] (-0.15,3.6)--(0.15,3.6);
       \draw[very thick] (1.6-0.15,2.4)--(1.6+0.15,2.4);
       \draw[very thick] (1.6-0.15,3.6)--(1.6+0.15,3.6);
       
       \draw[line width=3pt,red,dotted,opacity=0.8] (0,2.4)--(0,3.6);
        \draw[line width=3pt,blue,opacity=0.8] (1.6,2.4)--(1.6,3.6);
        \draw (0,2.8) circle(4pt) [fill];

       \draw (1.6, 3.3) circle(4pt) [fill] ;

       \node at (2.3, 3.3) {$(\by',y_{d})$};

       \node at (0.8,2.8) {$(\bx',x_{d})$};
       
        \node at (4.2,0.2) {$\R^{d-1}\times \{0\}$};

      \end{tikzpicture}

      \caption{Depiction of geometry  used in the proof of Theorem \ref{Trace-for-half-space}.}\label{fig:main}
      \end{figure}
The idea is again to split the left-hand side into three parts that can be controlled by the right hand side.  
As shown in Figure \ref{fig:main},
we choose $(\bx', x_{d}),(\bby,y_{d})\in\Om$ and rewrite 
\begin{align*}
u(\by',0)&= u(\by', 0)-u(\by',y_{d}) +u(\by',y_{d})\\
u(\bx',0)&=u(\bx',0)-u(\bx',x_{d})+u(\bx', x_{d})\,.
\end{align*}
Notice that the blue solid horizontal line and the red horizontal dashed line in Figure \ref{fig:main} show the possible
positions of
 $(\bx', x_{d})$ and $(\by',y_{d})$ respectively. 
The key is to determine the end points of these lines
 so that any $(\by', y_{d})$ over the blue solid line should stand in the effective neighborhood  (shown as red solid circle)  of 
any $(\bx', x_{d})$ on the red horizontal dashed line, in particular, the bottom end point whose
effective neighborhood is given by the dashed purple circle.
 Note that for any $y_{d} , x_{d}\in (0, M)$, we can write 
 \[
 \begin{split}
&|u(\bdx',0) - u(\bdy', 0)|^{p} \\
&\quad\leq 3^{p-1}\left( |u(\bdx', 0) - u(\bdx', x_{d})|^{p}  + |u(\bdy', y_{d}) - u(\bdx', x_{d})|^{p} +|u(\bdy', y_{d}) - u(\bdy', 0)|^{p} \right). 
\end{split}
 \]
 Let us denote ${\bf h} := \bdx'-\bdy'$. For  $1\leq \alpha < \beta \leq 2$ to be determined later, integrating first in $y_{d}$ in the interval $(\alpha|{\bf h}|, \beta|{\bf h}|)$, and then in the $x_d$ variable over the interval   $ (\alpha|{\bf h}|, \beta|{\bf h}|)$, we obtain that 
 \[
 \begin{split}
3^{1-p}|u(\bdx',0) - u(\bdy', 0)|^{p} &\leq \fint_{\alpha|{\bf h}|}^{\beta|{\bf h}|} |u(\bdx',0) - u(\bdx',x_{d} )|^{p} dx_{d}  + \fint_{\alpha|{\bf h}|}^{\beta|{\bf h}|}|u(\bdy',y_{d}) - u(\bdy',0)|^{p}dy_{d} \\
 &+ \fint_{\alpha|{\bf h}|}^{\beta|{\bf h}|}\fint_{ {\alpha|{\bf h}|}}^{\beta|{\bf h}|}|u(\bdy', y_{d}) - u(\bdx', x_{d})|^{p}dy_{d}d x_{d}. 
 \end{split}
 \]
 It follows then that 
 \[
 \begin{split}
 3^{1-p} I_{1}&\leq  \int_{\mathbb{R}^{d-1}}\int_{B_{\frac{M}{2}}(\bdx')} \fint_{\alpha|{\bf h}|}^{\beta|{\bf h}|} \frac{|u(\bdx',0) - u(\bdx',x_{d} )|^{p}} {|\bdx'-\bdy'|^{d + ps-2}}dx_{d}d\bdy'd\bdx'\\
  & + \int_{\mathbb{R}^{d-1}}\int_{B_{\frac{M}{2}}(\bdx')} \fint_{\alpha|{\bf h}|}^{\beta|{\bf h}|} \frac{|u(\bdy',y_{d}) - u(\bdy',0)|^{p}}{|\bdx'-\bdy'|^{d + ps-2}}dy_{d} d\bdy' d\bdx'\\
  &+  \int_{\mathbb{R}^{d-1}}\int_{B_{\frac{M}{2}}(\bdx')} \fint_{\alpha|{\bf h}|}^{\beta|{\bf h}|}\fint_{ {\alpha|{\bf h}|}}^{\beta|{\bf h}|}\frac{|u(\bdy', y_{d}) - u(\bdx', x_{d})|^{p}}{|\bdx'-\bdy'|^{d + ps-2}}dy_{d}d x_{d}d\bdy'd\bdx'.
 \end{split}
 \]
 It is not difficult to see that the first two integrals are equal   after a change of variables. 
 Therefore after rewriting it as $I_{1} \leq I_{11} + I_{12}$, we only need to estimate $I_{11}$ and $I_{12}$  which are given by  
 \[
 I_{11} = 2 \int_{\mathbb{R}^{d-1}}\int_{B_{\frac{M}{2}}(\bdx')} \fint_{\alpha|{\bf h}|}^{\beta|{\bf h}|} \frac{|u(\bdx',0) - u(\bdx',x_{d} )|^{p}} {|\bdx'-\bdy'|^{d + ps-2}}dx_{d}d\bdy'd\bdx'
 \]
 and 
 \[
 I_{12}=  \int_{\mathbb{R}^{d-1}}\int_{B_{\frac{M}{2}}(\bdx')} \fint_{\alpha|{\bf h}|}^{\beta|{\bf h}|}\fint_{ {\alpha|{\bf h}|}}^{\beta|{\bf h}|}\frac{|u(\bdy', y_{d}) - u(\bdx', x_{d})|^{p}}{|\bdx'-\bdy'|^{d + ps-2}}dy_{d}d x_{d}d\bdy'd\bdx'.
 \]
 To estimate $I_{11}$, we first  {change the integration with respect to $\bdy'$ to that with respect to ${\bf h} = \bdx'-\bdy'$ and obtain that}
 \[
 \begin{split}
 I_{11} &= 2 \int_{\mathbb{R}^{d-1}}\int_{B_{\frac{M}{2}}(\bdx')} \fint_{\alpha|{\bf h}|}^{\beta|{\bf h}|}\frac{|u(\bdx',x_{d}) - u(\bdx',0)|^{p}}{|\bx'-\by'|^{d + ps-2}}d x_{d} d\bdy'd\bdx'\\
 & =\frac{2 }{\beta-\alpha}  \int_{\mathbb{R}^{d-1}}\int_{B_{\frac{M}{2}}(\bm 0')} \int_{\alpha|{\bf h}|}^{\beta|{\bf h}|}\frac{|u(\bdx',x_{d}) - u(\bdx',0)|^{p}}{|{\bf h}|^{d + ps-1}}d x_{d} d{\bf h}d {\bdx'} .
 \end{split}\]
 Integrating the  {latter} using polar coordinates  {and letting $h = |{\bf h}|$,} we obtain that 
 and then iterating the integrals by change of variables, we obtain that 
 \[
 \begin{split}
I_{11} &=\frac{2\mathcal{H}^{d-2}(\mathbb{S}^{d-2})}{\beta-\alpha} \int_{\mathbb{R}^{d-1}}  \left(\int_{0}^{\frac{M}{2}} \int_{\alpha h}^{\beta h}\frac{|u(\bdx',x_{d}) - u(\bdx',0)|^{p}}{|h|^{d + ps-1}}h^{d-2} dx_{d} dh \right) d\bdy'd\bdx'\\
 &=\frac{2 \mathcal{H}^{d-2}(\mathbb{S}^{d-2})}{\beta-\alpha} \int_{\mathbb{R}^{d-1}}  \int_{0}^{M} \left(\int_{\frac{x_{d}}{\beta}}^{\frac{x_{d}}{\alpha}}\frac{h^{d-2}}{|h|^{d + ps-1}}dh \right) |u(\bdx',x_{d}) - u(0, \bdx')|^{p}dx_{d}  d\bdy'd\bdx'\\
&= \frac{2\mathcal{H}^{d-2}(\mathbb{S}^{d-2}) (\beta^{ps} - \alpha^{ps})}{ps(\beta-\alpha)} \int_{\mathbb{R}^{d-1}}  \int_{0}^{M} \frac{ |u(\bdx',x_{d}) - u(\bdx', 0)|^{p}}{|x_{d}|^{ps}}dx_{d} d\bdx'.\\
 \end{split}
 \]
Now a combination of Lemma \ref{1d-trace} and Lemma \ref{lemma-half-to-vartheta}  {implies that} there exists a constant $C$ independent of $\vartheta$ such that 
\[
\begin{split}
\int_{\mathbb{R}^{d-1}}\int_{0}^{M}&\frac{|u(\bdx',x_{d}) - u(\bdx',0)|^{p}}{|x_{d}|^{ps}}dx_{d} d\bdx'\\
& \leq C\left( \vartheta^{ps-p} \int_{\mathbb{R}^{d}_{M^+}} \int_{B_{\vartheta x_{d}}} \frac{|u(\bdy) - u(\bdx)|^{p}}{|\vartheta x_{d}|^{d+ ps}} d\bdy d\bdx +  M^{- ps} \|u\|_{L^{p}(\mathbb{R}^{d}_{M^+})}^{p}\right).
\end{split}
\]
Combining the above estimates, we obtain the desired estimate for $I_{11}$. 
Next we estimate $I_{12}$. Turns out that we will impose conditions on $\alpha$ and $\beta$ to get a control of  $I_{12}$.  To that end, recalling ${\bf h} = \bdx'-\bdy'$, we have 
 \[
\begin{split}
I_{12} &= \frac{1}{(\beta-\alpha)^{2}} \int_{\mathbb{R}^{d-1}}\int_{B_{M/2}({\bm 0'})} \int_{\alpha|{\bf h}|}^{\beta|{\bf h}|} \int_{\alpha|{\bf h}|}^{\beta|{\bf h}|}\frac{|u(\bdx' + {\bf h}, y_{d}) - u(\bdx', x_{d})|^{p}}{|{\bf h}|^{d + ps}}dy_{d}  dx_{d} d{\bf h} d\bdx'\\
& \leq \frac{ \beta^{d + ps}}{(\beta-\alpha)^{2}} \int_{\mathbb{R}^{d-1}}\int_{B_{M/2}({\bm 0'})} \int_{\alpha|{\bf h}|}^{\beta|{\bf h}|} \int_{\alpha|{\bf h}|}^{\beta|{\bf h}|}\frac{|u( \bdx' + {\bf h}, y_{d}) - u(\bdx', x_{d})|^{p}}{|x_{d}|^{d + ps}}dy_{d}  dx_{d} d{\bf h} d\bdx'
\end{split}
\]
where we used the fact that the integration the $x_{d}$ variable ranges from $\alpha |{\bf h}| $ to $ \beta |{\bf h}|$, and $\beta \leq 2$.   {Using polar coordinates for the integration on the $d-1$ dimensional ball $B_{\frac{M}{2}}(\bm 0')$}, we obtain that  
 {\small  \[
\begin{split}
I_{12} 
&\leq \frac{ \beta^{d + ps}}{(\beta-\alpha)^{2}} \int_{\mathbb{R}^{d-1}} \int_{\mathbb{S}^{d-2}} \left(\int_{0}^{\frac{M}{2}} \int_{\alpha h}^{\beta h} \int_{\alpha h}^{\beta h}\frac{|u(\bdx' + h \boldsymbol{\nu},y_{d} ) - u(\bdx', x_{d})|^{p}}{|x_{d}|^{d + ps}} h^{d-2}  dy_{d}  dx_{d} dh \right)  d\mathcal{H}^{d-2}(\boldsymbol{\nu}) d\bdx'
\end{split}
\]}
By iterating the integrals and using Fubini, we have that 
{\small \[
\begin{split}
I_{12} &\leq  \frac{\beta^{d + ps}}{(\beta-\alpha)^{2}} \int_{\mathbb{R}^{d-1}} \int_{\mathbb{S}^{d-2}}\left( \int_{0}^{M} \int_{\frac{x_{d}}{\beta}}^{\frac{x_{d}}{\alpha}} \int_{\alpha h}^{\beta h}\frac{|u( \bdx' + h \boldsymbol{\nu},y_{d} ) - u(\bdx',x_{d})|^{p}}{|x_{d}|^{d + ps}} h^{d-2}  dy_{d}  dh\right)  d\mathcal{H}^{d-2}(\boldsymbol{\nu})dx_{d} d\bdx'\\
&= \frac{\beta^{d + ps}}{(\beta-\alpha)^{2}} \int_{\mathbb{R}^{d-1}} \int_{0}^{M} \left(  \int_{\mathbb{S}^{d-2}}\int_{\frac{x_{d}}{\beta}}^{\frac{x_{d}}{\alpha}} \int_{\alpha h}^{\beta h}\frac{|u( \bdx' + h \boldsymbol{\nu},y_{d}) - u(\bdx',x_{d})|^{p}}{|x_{d}|^{d + ps}} h^{d-2}  dy_{d}  dhd\mathcal{H}^{d-2}(\boldsymbol{\nu})\right)  dx_{d} d\bdx'\\
&\leq \frac{\beta^{d + ps}}{(\beta-\alpha)^{2}} \int_{\mathbb{R}^{d-1}} \int_{0}^{M} \left(  \int_{\frac{x_{d}}{\beta} \leq |\bdy'-\bdx'| \leq \frac{x_{d}}{\alpha} } \int_{\alpha |\bdy'-\bdx'|}^{\beta |\bdy'-\bdx'|}\frac{|u(\bdy', y_{d}) - u(\bdx',x_{d})|^{p}}{|x_{d}|^{d + ps}}dy_{d}  d\bdy'\right)  dx_{d} d\bdx'.
\end{split}
\]}
Now for $\bdx = ( \bdx', x_{d}), \bdy = (\bdy', y_{d})\in \mathbb{R}^{d}$, with the property that  $y_{d} \in [\alpha |\bdy'-\bdx'|, \beta |\bdy'-\bdx'|]$ and $|\bdy'-\bdx'|  \in \left[\frac{x_{d}}{\beta}, \frac{x_{d}}{\alpha}\right]$, we have that, for 
$\theta_0^2=\frac{(\beta-\alpha)^{2} + 1}{\alpha^{2}}$,
\[
|\bdy - \bdx|^{2}  = (x_{d}- y_{d})^{2} + |\bdy' - \bdx'|^{2} \leq \max\left\{ \left(1 - \frac{\beta}{\alpha}\right)^{2} + \frac{1}{\alpha^{2}},  \left(1 - \frac{\alpha}{\beta}\right)^{2} + \frac{1}{\alpha^{2}}\right\} x_{d}^{2} \leq \theta_0^2 x_{d}^{2}
\]
since $\alpha <\beta$. Note that  for all $1 < \alpha < \beta \leq 2$, 
the quantity  
$\theta_0^2
\in (\frac{1}{2}, 2)$. We further choose $\alpha$ and $\beta$ such that $ 
\theta_0^2 \in (\frac{1}{2}, 1)$, which is possible,  
for example with $\beta =2$ and $\alpha = \frac{3}{2}$.
Using these values of $\alpha$ and $\beta$, we get
$\theta_{0}^{2}  
= \frac{5}{9}$, and
we can estimate $I_{12}$ as 
\[
I_{12} \leq C(d,p,s)\int_{\mathbb{R}^{d}_{M^+}} \int_{B_{\theta_{0} x_{d}}(\bdx) }\frac{|u(\bdy) - u(\bdx)|^{p}}{|x_{d}|^{d + ps}}d\bdy d\bdx.
\]
We now apply Lemma \ref{lemma-half-to-vartheta} corresponding to $\theta_{0} = \sqrt{5/9}$ to conclude that 
\[
I_{12} \leq  C  \left(\vartheta^{ps-p}\int_{\mathbb{R}^{d}_{M^+}} \int_{B_{\vartheta x_{d}} (\bdx)} \frac{|u(\bdy) - u(\bdx)|^{p} } {|\vartheta x_{d}|^{d + ps}}d\bdy d\bdx 
+M^{-ps} \|u\|^{p}_{L^{p}(\mathbb{R}^{d}_{M^+})}   \right)
\]
for any $\vartheta\in (0, 1].$
  That completes the proof. 
\end{proof}

\section{Results on Lipschitz domains}
\label{trace-bounded}
\subsection{Trace theorem for Lipschitz domains}
The aim of this subsection is the proof of Theorem \ref{coro:bdddomain-intro}. The proof follows classical arguments. We use the trace theorems on half space and stripes proved in the previous section to prove it for  Lipschitz hypographs and then for bounded domains with Lipschitz boundary.  
To be precise, following \cite{McLean},  we say that $\Omega\subset \mathbb{R}^{d}$ is a {\em Lipschitz hypograph} if there exists  a Lipschitz function $\zeta : \mathbb{R}^{d-1} \to \mathbb{R}$ and $L>0$ such that  $ \|\nabla \zeta\|_{L^{\infty}} \leq L$, 
\[
\Omega = \{\bdx = (\bdx', x_{d})\in \mathbb{R}^{d}:  x_{d} > \zeta(\bdx') \}, \,\,\text{and}\,\, \partial \Om = \{ \bdx = (\bdx', x_{d})\in \mathbb{R}^{d}: x_{d} = \zeta(\bdx') \}. 
\] 
We say that an open subset $\Om$ is  of $\R^d$ is a Lipschitz domain if its boundary $\partial \Omega$ is compact and there exists two families of sets $\{ O_{j}\}$ and $\{\Omega_{j}\}$ with properties: 
\begin{enumerate}
\item for each $j$, $O_j$ is bounded and the family $\{O_{j}\}$ is an open cover of $\partial \Omega$, 
\item each $\Omega_{j}$ can be transformed to a Lipschitz hypograph by a a change of coordinates (rigid motion), and 
\item the set $\Omega$ satisfies the property that $\Om\cap O_{j} = \Omega_{j}\cap O_{j}$ for all $j$.   
\end{enumerate}

\begin{proof}[Proof of Theorem \ref{coro:bdddomain-intro}]
We use standard procedures to prove the theorem. First, since $C^{0,1}_c(\overline{\Omega})$ is dense in $\Wf(\Omega)$ 
 it suffices to demonstrate that the linear operator $  {T}$  
 satisfies \eqref{trace-estimate} over $C^{0,1}_c(\overline{\Omega})$. Second, noticing that we have already proved the inequality \eqref{trace-estimate} when $\Omega = \mathbb{R}^{d}_{+}$, we begin by showing the validity of the theorem for Lipschitz hypographs, and finally using partition of unity and flattening argument we prove it for general Lipschitz domains. 
 
Step 1. We prove the statement when $\Omega $ is a Lipschitz hypograph.  { Let $\zeta$ be the associated Lipschitz function as above. Then the surface measure $d\sigma$ on $\partial \Omega$ is given by, \cite{McLean}, 
\[
d\sigma(\bx) = \sqrt{1 + |\nabla \zeta (\bx')|^2} d\bx'
\]
By definition $ u\in W^{s-{1\over p},p}(\partial \Omega)$ if and only if $u\in L^{p}(\partial \Omega, d\sigma)$ and 
\[
|u|^p_{W^{s-{1\over p},p}(\partial \Omega)}=\int_{\partial \Omega}\int_{\partial \Omega} \frac{|u(\by) - u(\bx)|^p}{|\bx-\by|^{d + ps-2}}d\sigma(\bx) d\sigma(\by) < \infty. 
\]
Now making the change of variables $\bx\mapsto \hat{\bx}=(\bx', \zeta(\bx'))$ and $\by\mapsto \hat{\by}=(\by',  \zeta(\by'))$ and using 
$u_{\zeta}(\bx'):= u(\bdx', \zeta(\bdx'))$, we have that  
{\small \[
 |u|^p_{W^{s-{1\over p},p}(\partial \Omega)}= \iint_{[\mathbb{R}^{d-1}]^2} \frac{|u_{\zeta}(\by')-u_{\zeta}(\bx')|^p}{(|\bx'-\by'|^2 + |\zeta(\by')-\zeta(\bx')|^{2})^{d+ps-2\over 2}}\sqrt{1 + |\nabla \zeta (\bx')|^2} \sqrt{1 + |\nabla \zeta (\by')|^2} d\by'd\bx'. 
\]}
Notice that for any $\bx',\by'\in \mathbb{R}^{d-1}$
\[
|\bx'-\by'|\leq \sqrt{|\bx'-\by'|^2 + |\zeta(\by')-\zeta(\bx')|^{2}} \leq \sqrt{1+ L^2}|\by'-\bx'| 
\]
and $1\leq \sqrt{1 + |\nabla \zeta(\bx')|^2} \leq \sqrt{1 + L^2}$. As a consequence, there is a constant $C=C(d, L)$ such that  \[{1\over C}|u|_{W^{s-{1\over p},p}(\partial \Omega)} \leq  |u_{\zeta}|_{W^{s-{1\over p}, p}(\mathbb{R}^{d-1})} \leq C |u|_{W^{s-{1\over p},p}(\partial \Omega)}\] 
We conclude that $u\in W^{s-{1\over p},p}(\partial \Omega)$ if and only if $u_\zeta \in W^{s-{1\over p},p}(\mathbb{R}^{d-1})$ and we may take the  the norm 
$\|u\|_{W^{s-{1\over p}, p}(\partial \Omega)} = \|u_{\zeta}\|_{W^{s-{1\over p},p}(\mathbb{R}^{d-1})}$. }

Now given $u\in  C^{0,1}_c(\overline{\Omega})$,  let us introduce the function $G_{\zeta}u (\bdx', x_{d}) := u(\bdx', x_{d}+ \zeta(\bdx')) $ for all $(\bdx',x_{d}) \in \mathbb{R}^{d}_{+}$. Then notice that $G_{\zeta}u$ has a compact support. Moreover, as a composition of Lipschitz functions,  $G_{\zeta}u\in C^{0,1}_{c}(\mathbb{R}^{d}_{+})$ and that 
\[
\|G_{\zeta}u\|_{L^{p}(\mathbb{R}^{d}_{+})}^{p}= \int_{\mathbb{R}^{d-1} }\int_{0}^{\infty}|u(\bdx', x_{d}+ \zeta(\bdx'))|^{p}d\bdx'dx_{d} = \|u\|^{p}_{L^{p}(\Omega)}
\]
We next show that there exists a $\sigma \in (0,1]$ and a positive constant $C$, depending only on the Lipschitz constant $L$, $d, p, $ $s$ such that $|G_{\zeta}u |_{\Wf_{\sigma}(\mathbb{R}^{d}_{+})} \leq C |u|_{\Wf(\Omega)}$. That is, 
\[
\int_{\mathbb{R}^{d}_{+}} \int_{B_{\sigma  x_{d}}(\bx)} \frac{|G_{\zeta}u(\bdy) - G_{\zeta}u(\bdx)|^{p}} {|\sigma x_{d}|^{{\mu} }} d\bdy d\bdx  \leq C \int_{\Omega} \int_{B_{\delta(\bdx)}(\bdx) } \frac{|u(\bdy) - u(\bdx)|^{p}} {| \delta(\bdx)|^{{\mu} }} d\bdy d\bdx, 
\]
where $\delta(\bdx) = \dist(\bdx, \partial \Omega).$
To that end, for a $\sigma$ to be determined, we estimate that 
\[
\begin{split}
\int_{\mathbb{R}^{d}_{+}} \int_{B_{\sigma  x_{d}} (\bdx) } &\frac{|G_{\zeta}u(\bdy) - G_{\zeta}u(\bdx)|^{p}} {|\sigma \cdot x_{d}|^{{\mu} }} d\bdy d\bdx\\
&=\int_{\mathbb{R}^{d}_{+} }\int_{B_{\sigma\cdot x_{d}}((\bdx', x_{d}))} \frac{ |u(\bdy', y_{d}+ \zeta(\bdy')) - u(\bdx', x_{d}+ \zeta(\bdx'))|^{p}}{ |\sigma\cdot x_{d}|^{{\mu}} } d\bdy'd y_{d} d\bdx \\
& =  \int_{\Omega }\int_{H_{\sigma}({\bf w})}\frac{ |u({\bf z}', z_{d}) - u({\bf w}', w_{d})|^{p}}{ |\sigma\cdot(w_{d}-\zeta({\bf w}'))|^{{\mu}} } d{\bf z}'d z_{d} d{\bf w} 
\end{split}
\]
where 
\[
H_{\sigma}({\bf w}) = \{({\bf z}', z_{d})\in \Omega: |{\bf z}'-{\bf w}'|^{2} + |z_{d}-w_{d} - ( \zeta({\bf w}' )-\zeta({\bf z}') )   |^{2} < (\sigma(w_{d} - \zeta ({\bf w}')) )^{2}  \}. 
\]
Note  that using the fact that $({\bf w}', \zeta({\bf w}')) \in \partial \Omega$, we have 
\begin{equation}\label{dist-from-boundary}
 {\dist}({\bf w},\partial \Omega)\leq | w_{d}-\zeta({\bf w}')| \leq K_{1}  {\dist}({\bf w},\partial \Omega)
\end{equation}
for some constant $K_{1}$ that depends only on the Lipschitz character of $\Omega$.  Indeed, the quantity $x_d -\zeta(\bdx')$ is represented in Figure \ref{fig:Lipgraph} below as the length of the read dashed line. 
In Figure  \ref{fig:Lipgraph}, $\bdw=(\bdw',w_d)$ is a point in $\Om$ and $(\bdw', \zeta(\bdw'))$ is a point on $\partial \Om$.  The Lipschitz graph $\zeta$ 
always remains outside the double cone centered at $(\bdw', \zeta(\bdw'))$. It is obvious from Figure \ref{fig:Lipgraph} that 
dist$(\bdw, \partial\Om)$ is greater than or equal to the length of black dashed line, which represents the distance of $\bdw$ to the edge of the cone. 
Therefore, $ {\dist}(\bdw, \partial\Om)\geq c (w_d -\zeta(\bdw'))$ where $c>0$ is a constant that only depends on $L$. \begin{figure}[htbp] 
\centering
  \begin{tikzpicture}[scale=1.2]  
    \tikzset{to/.style={->,>=stealth',line width=1pt}}; 
        \draw[to] (0,0)--(5,0) node[xshift=.3cm] {$\bdw'$};;
           \draw[to] (0,0)--(0,4) node[yshift=.3cm] {$w_d$};
     \draw[line width=1pt] (-1,1)--(1,2)--(2,1.5)--(4,2.2)--(5,1.5);       
     \node at (3.2, 1.6) {$\zeta(\bdw')$};
         \node at (5.3, 1.6) {$\partial\Om$};
       \draw (2.5,3) circle(2pt) [fill] node[above] {$(\bdw',w_d)$};
    \draw[red, dashed, line width=2pt]   (2.5,3)--(2.5, 1.6);
      \draw[dashed, line width=2pt]   (2.5,3)--(3.1, 2.3);
    \draw[blue, line width=1pt] (1, 0.4)--(2.5, 1.7)--(4, 3);
    \node at (4.8, 3) {\small slope $= L$};
      \node at (4.8, 0.4) {\small slope $= -L$};
        \draw[blue, line width=1pt] (1, 3)--(2.5, 1.7)--(4, 0.4);
      \end{tikzpicture}
 \caption{\small The Lipschitz graph $\zeta$ remains outside the double cone centered at $(\bdw', \zeta(\bdw'))$. The double cone is generated by the blue lines with axis parallel to the $w_d$-axis. The red dashed line represents $w_d - \zeta(\bdw')$ and the black dashed line represents the distance of $(\bdw', w_d)$ to the edge of the cone.} 
 \label{fig:Lipgraph}     
\end{figure}
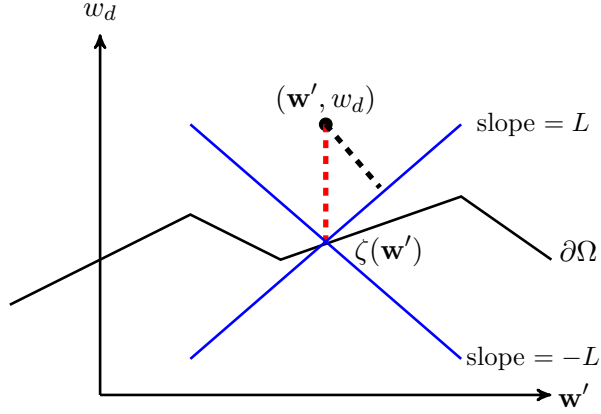

\noindent Moreover, for any ${\bf z} \in H_{\sigma}({\bf w})$, we have 
\[
\begin{split}
|{\bf z} - {\bf w}|^{2} &\leq 2(| z_{d}-w_{d} - ( \zeta({\bf w}' )-\zeta({\bf z}') )|^{2} + | \zeta({\bf w}' )-\zeta({\bf z}')|^{2}) + |{\bf w}' - {\bf z}'|^{2}\\
&\leq (\sigma K _{2} \dist({\bf w}, \partial \Omega))^{2},
\end{split}
\]
where $K_{2}$ is another universal constant depending only on the Lipschitz constant.  We choose $\sigma\in (0, 1)$ small so that $\sigma K _{2} < 1$. With this choice of $\sigma$ we have  $ H_{\sigma}({\bf w})\subset B_{\delta(\bdw)} ({\bf w})$. 
It then follows that there exists a constant $C$ such that 
\[
\int_{\mathbb{R}^{d}_{+}} \int_{B_{\sigma \cdot x_{d}}(\bdx)} \frac{|G_{\zeta}u(\bdy) - G_{\zeta}u(\bdx)|^{p}} {|\sigma \cdot x_{d}|^{{\mu} }} d\bdy d\bdx \leq C  \int_{\Omega }\int_{B_{\delta(\bdw)} ({\bf w})}\frac{ |u(\bdz) - u(\bdw)|^{p}}{ |\delta(\bdw)|^{\mu} } d\bdz d{\bf w}. 
\]
Applying  {Remark \ref{rem:trace_half_space} with $\vartheta=\sigma$,} there exists a constant $C = C(d, p, s)$ such that 
\[
\begin{split}
\| {T}G_{\zeta}u\|^{p}_{W^{s-\frac{1}{p}} (\mathbb{R}^{d-1})} &\leq C \left( \|G_{\zeta}u \|_{L^{p}(\mathbb{R}^{d}_{+})}^{p} + \sigma ^{ps-p} \int_{\mathbb{R}^{d}_{+}} 
\int_{B_{\sigma \cdot x_{d}(\bdx)} }\frac{|G_{\zeta}u(\bdy) - G_{\zeta}u(\bdx)|^{p}} {|\sigma \cdot x_{d}|^{{\mu} }} 
d\bdy d\bdx\right ) \\
&\leq C \left( \|u\|_{L^{p}(\Omega)}^{p} +  \int_{\Omega }\int_{B_{\delta(\bdw)} ({\bf w})}\frac{ |u(\bdz) - u(\bdw)|^{p}}{ |\delta(\bdw)|^{\mu} } d\bdz d{\bf w} \right) = C \|u\|_{\Wf(\Omega)}^{p}. 
\end{split}
\]
We finally observe that $ {T}G_{\zeta}u(\cdot) = G_{\zeta} u(\cdot, 0) = u_{\zeta} (\cdot)$, therefore, 
\[
\| {T}u\|^{p}_{W^{s-\frac{1}{p}} (\partial \Omega)}  = \|u_{\zeta}\|^{p}_{W^{s-\frac{1}{p}}(\mathbb{R}^{d-1}) } = \| {T}G_{\zeta}u\|^{p}_{W^{s-\frac{1}{p}} (\mathbb{R}^{d-1})}  \leq C \|u\|_{\Wf(\Omega)}^{p}
\]

Step 2. Assume now that $\Omega$ is a Lipschitz domain. Let the family of sets $\{\Omega_{i}\}_{i=1}^{N}$ and $\{O_{i}\}_{i=1}^{N}$ are as given at the beginning of the section. We introduce a partition of unity $\{ \phi_{i}\}$ subordinate to $\{O_{i}\}_{i=1}^{N}$.  By Lemma \ref{multibysmoothfn} and the property that $\Omega\cap O_{i} = \Omega_{i}\cap O_{i}$, the functions $u_{i} = \phi_{i} u$, after extending by $0$, are in $C^{0,1}_c(\overline{\Omega_{i}})$ for all $i=1, 2, \dots, N$ and that there exists a constant  $C$ such that 
\begin{equation}\label{truncated-u-control}
\|u_{i}\|_{\Wf(\Omega_{i})} \leq C \|u\|_{\Wf(\Omega)}, \text{ for all $ i=1, 2, \dots, N$}. 
\end{equation}
 It is also clear that 
\[
 {T}u (\xi) = u(\xi) = \sum_{i=1}^{N} u_{i}(\xi) = \sum_{i=1}^{N}  {T}^iu_{i}(\xi), \quad \text{for all $\xi \in \partial\Omega$}. 
\] 
Here we are the using the notation $ {T}^i u_i (\xi) =  {T}u_i(\xi) \chi_{\partial\Omega_i\cap \partial \Omega} (\xi)$,  the restriction of the trace map $ {T}$ on $\partial\Omega_i\cap \partial \Omega$.  Notice that for each $i$
$
\text{supp}( {T}^i u_i)\subset \partial\Omega\cap \text{supp}(\phi_i)
$ 
and as a consequence for each $i$
\begin{equation}\label{dist-supp}
 {\dist}(\text{supp}( {T}^i u_i), \partial \Omega\setminus \partial \Omega_i) > 0. 
\end{equation}
Now we  may write 
\begin{equation}\label{trace-decomposition}
\| {T}u\|_{W^{s-\frac{1}{p}} (\partial\Omega)} \leq \sum_{i=1}^{N} \| {T}^i u_{i}\|_{W^{s-\frac{1}{p},p} (\partial \Omega)}. 
\end{equation}
  For each $i$, 
  \[
  \begin{split}
  \| {T}^iu_{i}\|^{p}_{W^{s-\frac{1}{p},p} (\partial \Omega)} &= \int_{\partial \Omega}| {T}^iu_{i} (\xi)|^{p} d\sigma(\xi) + \int_{\partial \Omega}\int_{\partial \Omega} \frac{| {T}^iu_{i}(\xi) -  {T}^iu_{i}(w)|^{p} }{|\xi -w|^{d + ps -2}}  {d\sigma(\xi) d\sigma(w)}\\
  &\leq C \left(\int_{\partial \Omega_{i}}| {T}u_{i} (\xi)|^{p} d\sigma(\xi) + \int_{\partial \Omega_{i}}\int_{\partial \Omega_{i}} \frac{| {T}u_{i}(\xi) -  {T}u_{i}(w)|^{p} }{|\xi -w|^{d + ps -2}}  {d\sigma(\xi) d\sigma(w)}\right)\\
  &=C \| {T} u_{i}\|_{W^{s-\frac{1}{p},p}(\partial \Omega_{i})}^{p} 
  \end{split}
  \]
  where we used the observation in \eqref{dist-supp} about the support of  $ {T}^i u_{i}$ and its a positive distance away from the boundary $ \partial \Omega\setminus \partial \Omega_i$.   
  Since we know that $\partial \Omega_{i}$ is a hypograph, up to a rigid transformation, we  can apply Step 1, to write the estimate that there exists a constant $C$ such that for all $i = 1, 2, \dots, N$ 
  \begin{equation}\label{Step1-for-Omi}
   \| {T} u_{i}\|_{W^{s-\frac{1}{p},p}(\partial \Omega_{i})} \leq C \|u_{i}\|_{\Wf(\Omega_{i})}. 
  \end{equation}
  We finally get the inequality after we put together and \eqref{truncated-u-control}, \eqref{trace-decomposition} and \eqref{Step1-for-Omi}. That completes the proof. 
  \end{proof}
  \subsection{Hardy-type inequality for {Lipschitz} domains}
 We now show  {the proof of Theorem \ref{NLHardy-Lips}}. 
 The theorem is an immediate consequence of two results that we prove below. The first result establishes equation \eqref{prop:ghardy-semi} with a right hand side that is a norm instead of the seminorm.  {This result is then used to show a nonlocal Poincar\'e-type inequality (Proposition \ref{Poincare}). Propositions \ref{hardy-for-domain} and \ref{Poincare} together imply Theorem \ref{NLHardy-Lips}.}
 
\begin{prop}\label{hardy-for-domain}
Let $1\leq p < \infty$, $s\in (0, 1]$ and $ps > 1$. Assume that $\Omega$ is a bounded Lipschitz domain in $\mathbb{R}^{d}(d\geq 2)$.
Then there exists a constant $C>0$ such that for any $u\in \mathfrak{\mathring{W}}^{p,s}(\Omega)$
\beq \label{prop:ghardy}
 \int_\Om \frac{|u(\bm x)|^p}{( {\dist}(\bm x, \partial \Om))^{ps}} d\bm x \leq C  \|u\|^p_{\Wf (\Omega)}\,.
 \eeq
\end{prop}
\begin{proof}
 {Without loss of generality we assume that $u\in C_{c}^{0, 1}(\Omega)$.} 
Following the proof of Theorem \ref{coro:bdddomain-intro}, we show \eqref{prop:ghardy} in two steps. 

Step 1. We prove the statement when $\Omega $ is a Lipschitz hypograph, i.e, 
$\Omega= \{ (\bdx', x_{d}) : x_{d} > \zeta(\bdx')\} $
for a Lipschitz function $\zeta:\R^{d-1}\to\R$ with Lipschitz constant $L$. 
 {
Introduce the function $G_{\zeta}u (\bdx', x_{d}) := u(\bdx', x_{d}+ \zeta(\bdx')) $ for all $(\bdx',x_{d}) \in \mathbb{R}^{d}_{+}$ as before. Then $G_{\zeta}u \in C_c^{0, 1}(\mathbb{R}^{d}_{+})$
and applying Remark \ref{rem:hardy-for-smooth-special}, for any $\sigma\in (0, 1]$,  we have 
\[
\int_{\mathbb{R}^{d}_{+}}\frac{|G_{\zeta} u(\bdx)|^{p}}{|x_{d}|^{ps}} d\bdx 
\leq C |G_{\zeta}u |^p_{\Wf_{\sigma}(\mathbb{R}^{d}_{+})} \,,
\]
where $C$ is a constant depending only on $d$, $p$, $s$ and $\sigma$.  
From Step 1 in the proof of Theorem \ref{coro:bdddomain-intro} that  $|G_{\zeta}u |_{\Wf_{\sigma}(\mathbb{R}^{d}_{+})} \leq C |u|_{\Wf(\Omega)}$ for some $\sigma \in (0,1]$, and so we have  
\beq \label{eq:ghardy_step1_1}
\int_{\mathbb{R}^{d}_{+}}\frac{|G_{\zeta} u(\bdx)|^{p}}{|x_{d}|^{ps}} d\bdx \leq C |G_{\zeta}u |^p_{\Wf_{\sigma}(\mathbb{R}^{d}_{+})}\leq C |u|^p_{\Wf(\Omega)}\,. 
\eeq }
Now by change of variables we can see 
\[
\int_{\mathbb{R}^{d}_{+}}\frac{|G_{\zeta} u(\bdx)|^{p}}{|x_{d}|^{ps}} d\bdx = \int_{\Om}\frac{ |u(\bdx)|^{p}}{|x_d -\zeta(\bdx')|^{ps}} d\bdx \,.
\]
Using the inequality \eqref{dist-from-boundary}, see also  Figure \ref{fig:Lipgraph}, we then have 
\beq \label{eq:ghardy_step1_2}
\int_{\mathbb{R}^{d}_{+}}\frac{|G_{\zeta} u(\bdx)|^{p}}{|x_{d}|^{ps}} d\bdx = \int_{\Om}\frac{ |u(\bdx)|^{p}}{|x_d -\zeta(\bdx')|^{ps}} d\bdx \geq c^{ps} \int_{\Om}\frac{ |u(\bdx)|^{p}}{| {\dist}(\bdx, \partial\Om)|^{ps}} d\bdx \,.
\eeq 
By combining \eqref{eq:ghardy_step1_1} and \eqref{eq:ghardy_step1_2}, we have shown \eqref{prop:ghardy} for $\Om$ being a Lipschitz hypograph. 

Step 2. Assume now that $\Omega$ is a Lipschitz domain. Let the family of sets $\{\Omega_{i}\}_{i=1}^{N}$ and $\{O_{i}\}_{i=1}^{N}$ are as given at the beginning of the section.  
Notice that  $\Om\cap O_i =\Om_i \cap O_i$ and each $\Omega_{i}\, (i\in \{ 1,\cdots, N\})$ can be transformed to a Lipschitz hypograph by rigid motion. In addition, we define an open set $O_0\subset \Om$ such that $\{ O_i\}_{i=0}^N$ is an open cover of $\overline{\Om}$. 
Without loss of generality, we assume that $ {\dist}(O_0, \partial\Om)= c^\ast >0$. 
Similarly, we introduce a partition of unity $\{ \phi_i\}_{i=0}^N$ for $\overline{\Om}$ subordinate to $\{ O_i\}_{i=0}^N$. 
Therefore, for any $\bdx\in \overline{\Om}$, $\textstyle u(\bdx) = \sum_{i=0}^N (\phi_i u)(\bdx)$.
We also define $u_i = \phi_i u$ and assume a zero extension of $u_i$ outside $O_i$ when it is necessary. 
Then 
\beq \label{eq:ghardy_step2_1}
\int_{\Om}\frac{ |u(\bdx)|^{p}}{| {\dist}(\bdx, \partial\Om)|^{ps}}d\bdx \leq (N+1)^{p-1} \sum_{i=0}^N \int_{\Om\cap O_i} \frac{ |u_i(\bdx)|^{p}}{| {\dist}(\bdx, \partial\Om)|^{ps}} d\bdx\,.
\eeq 
Notice that on $O_0$ we have 
\beq \label{eq:ghardy_step2_2}
\int_{O_0} \frac{ |u_0(\bdx)|^{p}}{| {\dist}(\bdx, \partial\Om)|^{ps}} d\bdx\leq (c^\ast)^{-ps} \int_{O_0} |u_0(\bdx)|^{p} \leq C \| u \|^p_{L^p(\Om)}\,. 
\eeq
For any $\bdx\in O_i\, (i\in \{1,\cdots, N\})$, we could assume without loss of generality that $ {\dist}(\bdx, \partial\Om)=  {\dist}(\bdx, \partial\Om_i)$
(which is true by choosing small enough sets $\{ O_i\}_{i=1}^N$). Then for $i\in \{1,\cdots, N\}$,
\beq \label{eq:ghardy_step2_3}
\int_{\Om\cap O_i} \frac{ |u_i(\bdx)|^{p}}{| {\dist}(\bdx, \partial\Om)|^{ps}} d\bdx \leq \int_{\Om_i} \frac{ |u_i(\bdx)|^{p}}{| {\dist}(\bdx, \partial\Om_i)|^{ps}} d\bdx \leq |u_i |^p_{\Wf (\Omega_i)}\,,
\eeq
where we have used Step 1 on $\Om_i$, since it is a Lipschitz hypograph up to a rigid transformation. 
 By combining equations \eqref{truncated-u-control}, \eqref{eq:ghardy_step2_1}, \eqref{eq:ghardy_step2_2}, and \eqref{eq:ghardy_step2_3}, we
 have shown \eqref{prop:ghardy} for a general Lipschitz domain $\Om$.  
\end{proof}

\subsection{A Poincar\'e-type inequality for {Lipschitz} domains}
\begin{prop}[Poincar\'e-type inequality]\label{Poincare}
Let $1<p < \infty$, $s\in (0, 1]$ and $ps > 1$. Assume that $\Omega$ is a bounded Lipschitz domain in $\mathbb{R}^{d}(d\geq 2)$.
There exists a constant $C_P>0$ such that 
\[
\|u\|_{L^{p}(\Omega)} \leq C_P |u|_{\mathfrak{W}^{s,p}(\Omega)} 
\]
for all $u\in \mathfrak{{\mathring W}}^{s, p}(\Omega)$
\end{prop}
\begin{proof}
Suppose the inequality is false. Then there exists a sequence $\{u_n\} \in \mathfrak{{\mathring W}}^{s,p}(\Omega)$ such that $\|u_n\|_{L^p}=1$ for all $n$ and $|u_n|_{\mathfrak{W}^{s, p}(\Omega)} \to 0$ as $n\to \infty.$ 
Since $\{u_n\}$ is a bounded sequence in $\mathfrak{\mathring{W}}^{s, p}(\Omega)$, we have a weakly convergent subsequence with a limit $u \in\mathfrak{\mathring{W}}^{s, p}(\Omega)$. Moreover, this $u$ is also an $L^p$ weak limit of $\{u_n\}$ since $L^{p}(\Omega)$ is a subset of the dual of $\mathfrak{W}^{s, p}(\Omega)$.    {The seminorm $|\cdot|_{\mathfrak{\mathring{W}}^{s, p}(\Omega)}$ {,} being convex and (strongly) continuous ,  is thus weakly lower semicontinuous}. Therefore we have 
\[
|u|_{\mathfrak{W}^{s,p }(\Omega)} \leq \liminf_{n\to \infty}|u_n|_{\mathfrak{W}^{s, p}(\Omega)}=0
\]
 {As a consequence,} $u$ is a constant  {on each connected component of $\Omega$}. In particular, $u=0$ since the only  {(piecewise)} constant in $\mathfrak{\mathring{W}}^{s, p}(\Omega)$ is 0. Our next goal is to show that $u_n\to 0$ strongly in $L^{p}(\Omega)$ (up to a subsequence). To show this let $\epsilon >0$ be fixed.  {Let $\Upsilon^{m}_p= \min\{m, \gamma_p\}$, where $m>0$ and $\gamma_p=\gamma_p^1$ is defined by \eqref{symmetrized-kernel} with $\vartheta=1$. } 
Then 
\[
j(\bx) = {1\over 2}\int_{\Omega}\Upsilon^{m}_p(\bx, \by)d\by
\] 
has a minimum $c>0$ on $\Omega_{\epsilon}=\{\bx\in\Omega:  {\dist}(\bx, \partial \Omega)>\epsilon \}$. 
It then follows from the symmetricity $\gamma_p$ and the algebraic inequality $|b|^p\geq |a|^{p} + p|a|^{p-2}a(b-a)$, which holds true for $p\geq 1,$ that 
\[
\begin{split}
&|u_n|_{\mathfrak{W}^{s, p}(\Omega)}^{p} = {1\over 2}\int_{\Omega}\int_{\Omega} \gamma_p(\bx, \by) |u_n(\bx)-u_{n}(\by)|^{p}d\by d\bx\\
&\geq {1\over 2}\int_{\Omega}\int_{\Omega} \Upsilon^{m}_p(\bx, \by) |u_n(\bx)-u_n(\by)|^{p}d\by d\bx\\
&\geq{1\over 2} \int_{\Omega}\left(\int_{\Omega} \Upsilon^{m}_p(\bx, \by) d\by\right) |u_n(\bx)|^pd\bx - {p\over 2}\int_{\Omega}\left(\int_{\Omega} \Upsilon^{m}_p(\bx, \by)u_n(\by)d\by\right) |u_n(\bx)|^{p-2}u_{n}(\bx)d\bx\\
&\geq  \int_{\Omega_\epsilon} j(\bx) |u_n(\bx)|^pd\bx - {p\over 2}\int_{\Omega}\mathcal{K}u_n(\bdx) |u_n(\bx)|^{p-2}u_{n}(\bx)d\bx
\end{split}
\]
where the linear operator  $\mathcal{K}$ on $L^{p}(\Omega)$ is given by $\mathcal{K}u(\bx)=\int_{\Omega} \Upsilon^m_p(\bx, \by)u(\by)d\by$. Now by restricting the first integral in the right hand side on $\Omega_\epsilon$ we have for each $n$ 
\[
|u_n|_{\mathfrak{W}^{s, p}(\Omega)}^{p}\geq c\|u_n\|^p_{L^{p}(\Omega_{\epsilon})} - {p\over 2}\int_{\Omega}\mathcal{K}u_n(\bdx)|u_n(\bx)|^{p-2}u_{n}(\bx)d\bx.
\] 
We note that the integral operator $\mathcal{K}$ on $L^{p}(\Omega)$ is generated by the kernel $\Upsilon^m_p(\bx, \by)$ with finite double-norm in $L^{p}(\Omega)$, i.e. the map $\bx\mapsto\int_{\Omega} |\Upsilon^m_p(\bx, \by)|^{p'}d\by \in L^{p}(\Omega)$ where $p'$ is the H\"older conjugate of $p$. Such types of operators are known to be compact operators on $L^{p}(\Omega)$  \cite{Grobler, Luxemburg-Zaanen}. As a consequence, since we have show that $u_n\rightharpoonup 0,$ weakly in $L^{p}(\Omega)$, then $\mathcal{K}u_{n}\to 0$ strongly in $L^{p}(\Omega).$ Moreover, since $\{ |u_n|^{p-2}u_{n}\}$ is a bounded sequence in $L^{p'}(\Omega)$, we have 
\[
\int_{\Omega}\mathcal{K}u_n(\bdx)|u_n(\bx)|^{p-2}u_{n}(\bx)d\bx \to 0, \,\,\text{as $n\to \infty.$}
\]
Therefore we have 
\[
\limsup_{n\to \infty}\| u_n\|_{L^{p}(\Omega_\epsilon)} \leq {1\over \sqrt[p]{c}}\limsup_{n\to \infty}| u_n|_{\mathfrak{W}^{s,p}(\Omega)}=0. 
\]
We next study the behavior of the sequence on near the boundary on $\Omega\setminus \Omega_\epsilon$. To that end, for any $x\in\Omega\setminus \Omega_\epsilon$, $ {\dist}(x, \partial \Omega) \leq \epsilon$ and so applying Proposition \ref{hardy-for-domain}
we have for a uniform positive constant $C>0$
\[
{1\over \epsilon^{ps}} \int_{\Omega\setminus \Omega_\epsilon}|u_n|^{p}d\bx  \leq \int_{\Omega}{|u_n(\bx)|^{p}\over  {\dist}(\bx, \partial \Omega)^{ps}} d\bx \leq C\| u_n\|^p_{\mathfrak{W}^{s, p}(\Omega)}.
\]
Combining the estimates on $\Om_\ep$ and $\Om\backslash\Om_\ep$,  we have \[ \limsup_{n\to\infty} \| u_n\|_{L^p(\Om)}\leq C \ep^s\]for any $\ep>0$.
Thus $\limsup_{n\to\infty}\|u_n\|_{L^{p}(\Omega)}=0$, which is the contradiction we were looking for. 
\end{proof}
%

We note that the Poincar\'e-type inequality showing here extends a result stated in \cite{TTD19} for the case of $p=2$ and  {$s=1$} and also fills in a gap in the proof presented there.

\section{application: a variational problem}\label{appl}
In this section and as we discussed in the introduction, we present  an application of the trace theorem that study the problem of minimizing an energy functional  defined 
over a convex and closed subset of  $
\mathfrak{{\mathring W}}^{s, p}(\Omega)$.  The main objective of this section is proving Theorem \ref{obstacle}.  The proof is based on the direct method of calculus of variation \cite[Theorem 1.2]{DCoV}.  {From  the fact that $\mathfrak{W}^{s, p}(\Omega)$ is a reflexive Banach space, as mentioned in Section \ref{Prelim}, we know that $\mathfrak{{\mathring W}}^{s, p}(\Omega) \subset \mathfrak{W}^{s, p}(\Omega)
$ is also a reflexive Banach space.} 
As a consequence, energy functionals on $\mathfrak{{\mathring W}}^{s, p}(\Omega)
$ will attain their minimum in a weakly closed subset provided they are coercive
and (sequentially) weakly lower semi-continuous on such a subset with respect to $
\mathfrak{{\mathring W}}^{s, p}(\Omega)
$.

\begin{proof}[Proof of Theorem \ref{obstacle} and Corollary \ref{obstaclep=2}]
We first observe that  {for $f\in [\mathfrak{W}^{s,p}(\Omega)]^\ast$}, if $u$ is a minimizer of $E(v) - \langle f, v \rangle$ in $K_{\phi}(p, h)$,  then $w=u-\phi$ is a minimizer of 
$
E(u+\phi) -\langle f, u+\phi \rangle
$ in   {$K_{0}(p, h-\phi)$} and vice versa.  So we will focus on the latter energy. 
Observe that since  {$K_{0}(p, h-\phi)$} is convex and closed in the strong topology, then it is a  weakly closed subset of $
\mathfrak{{\mathring W}}^{s, p}(\Omega)
$. Now for  {$f\in [\mathfrak{{W}}^{s, p}(\Omega)]^\ast$}, consider the functional 
\[
E_{f}(u) = E(u+\phi) - \langle f, u +\phi \rangle.
\]
where $E$ is as given by \eqref{Energy}. 
Let us show this functional is weakly lower semicontinuous. Since $u\mapsto \langle f,u+\phi\rangle$ is weakly continuous, it suffices to show that $E(\cdot + \phi)$ is weakly lower semicontinuous. Suppose that $u_n\to u$ weakly in $\mathfrak{W}^{s,p}(\Omega)$. Using the inequality 
\[
F(\tau)-F(t) \geq F'(t)(\tau-t),\quad \text{for all $t,\tau\in \mathbb{R}$}
\]
which follows from the convexity of $F$, we have the inequality 
\begin{equation}\label{Energy-convexity-inequality}
E(u_n+\phi) -E(u+\phi) \geq   \int_{\Omega}\int_{\Omega} {A(\bx, \by)\over \delta(\bx)^{\mu}} F'(u_\phi(\bx)-u_\phi(\by)) (u_n(\by)-u_n(\bx) - (u(\by) - u(\bx)))d\by d\bx  
\end{equation}
where $u_\phi = u + \phi\in \mathfrak{W}^{s,p}(\Omega). $
The expression in the right hand side can be rewritten as the action of the functional $\Phi(u)$ on the difference $u_n-u$ where for $u\in \mathfrak{W}^{s,p}(\Omega)$
 functional $\Phi(u)$ on  $\mathfrak{W}^{s,p}(\Omega)$ is defined as 
\[
\langle \Phi(u), v\rangle =  \int_{\Omega}\int_{\Omega} {A(\bx, \by)\over \delta(\bx)^{\mu}} F'(u_\phi(\bx)-u_\phi(\by)) (v(\by) - v(\bx))d\by d\bx.   
\]
$\Phi(u)$ is in fact in the dual space of $\mathfrak{W}^{s,p}(\Omega)$. Indeed, we only check its boundedness. For any $v\in \mathfrak{W}^{s,p}(\Omega)$, by the bound on the growth of $F'$ in \eqref{Convexity-of-F} and H\"older's inequality we have we have 
\[
\begin{split}
|\langle \Phi(u), v\rangle| &\leq c_1\int_{\Omega}\int_{\Omega} {A(\bx, \by)\over \delta(\bx)^{\mu}} |u_\phi(\bx)-u_\phi(\by)|^{p-1} |v(\by) - v(\bx)|d\by d\bx \\
&\leq c_1|u_\phi|_{\mathfrak{W}^{s,p}(\Omega)}^{p-1}|v|_{\mathfrak{W}^{s,p}(\Omega)}. 
\end{split}
\]Thus we can write \eqref{Energy-convexity-inequality} as 
\[
E(u_n) -E(u) \geq  \langle \Phi(u), u_n-u\rangle. 
\]
Now taking the liminf on both sides of the equation and using the weak convergence of $u_n$ to $u$ in $\mathfrak{W}^{s, p}(\Omega)$ we have that 
\[
E(u) \leq \liminf_{n\to \infty}E(u_n).
\]
Moreover, $E_f$ is coercive over $\mathfrak{{\mathring W}}^{s, p}(\Omega)$. In fact, from duality and the bound for $F$ from \eqref{Convexity-of-F} we have 
that {for any $u\in\mathfrak{{\mathring W}}^{s, p}(\Omega)$,}
\[
E_f(u) \geq \alpha_1c_1|u_\phi|^p_{\mathfrak{W}^{s, p}(\Omega)}- 
{
\|f\|_{
[\mathfrak{{W}}^{s, p}(\Omega)]^\ast}
\|u\|_{\mathfrak{W}^{s, p}(\Omega)} + \langle f, \phi\rangle
}
\]
Now using the trivial inequality $|u|^p_{\mathfrak{W}^{s,p}(\Omega)} \leq 2^{p-1}| u_\phi|^p_{\mathfrak{W}^{s, p}(\Omega)} +2^{p-1}|\phi|^p_{{\mathfrak{W}^{s, p}(\Omega)}}$, the Poincar\'e-type inequality,  {and Young's inequality for products}
\[
\begin{split}
E_f(u) &\geq c\|u\|^{p}_{\mathfrak{W}^{s, p}(\Omega)}
- C |\phi|^p_{{\mathfrak{W}^{s, p}}} 
 {-C_\varepsilon
   \| f\|_{[{\mathfrak{W}}^{s,p}]^\ast
  }^{p/(p-1)}-\varepsilon \| u \|_{\mathfrak{W}^{s,p}}^p}
+ \langle f, \phi\rangle\\
\end{split}
\]
 {
Let $\varepsilon=c/2$ in the above, we have
\[
E_f(u) \geq \frac{c}{2}\|u\|^{p}_{\mathfrak{W}^{s, p}(\Omega)} -C
\]
where the constants $c$ and $C$ are independent of $u$.}
Therefore, \cite[Theorem 1.2]{DCoV} is applicable to conclude that $E_f$ has attains its minimum in  {$K_{0}(p, h-\phi)$}.

The proof of Corollary \ref{obstaclep=2} follows from the above and 
the theorem of Stampacchia \cite[Theorem 5.6]{brezisbook}.  {In particular, corresponding to $f\in [\mathfrak{W}^{s, 2}(\Omega)]^\ast$, a unique solution $u\in K_\phi(2,h)$ exists satisfying the inequality \eqref{obstacle-inequality}}. 
\end{proof}

 \section{Conclusion}
 
 This work is a continuation of earlier studies on nonlocal models involving nonlocal interactions with a varying horizon. A major contribution is to extend existing results to more general and non-Hilbert space settings.
 This can be very useful in treating nonlinear problems as illustrated here through an application to nonlinear obstacle problem. The theory can also be useful in the design and analysis of numerical approximations of the variational problems \cite{TiDu14}.
 The main results here provide another
demonstration to 
 the regularity pick-up  due to the vanishing nonlocal horizon so that local boundary conditions can be imposed for nonlocal variational problems.
We anticipate that further extensions can be explored, such as analogous results for spaces involving more general localization strategies (instead of having the horizon parameter linearly proportional to the distance function to the boundary). For applications in mechanics, we may extend the study to spaces of vector fields, for example, by introducing heterogeneous localizations to spaces studied in \cite{MeDu14,MeDu15} and nonlocal analog of spaces with only control of local divergence rather than full gradient.  Moreover, it is interesting to explore systematically how the possibly heterogeneous spatial nonlocal interactions over a given domain can induce the effective interactions over a (possibly lower dimensional) subset.  The latter can lead to more mathematical studies as well as practical applications.\\

{\bf Acknowledgement.} The authors would like to thank the referees for their careful reading of the manuscript and  their constructive comments that helped improving the presentation of the work.

\bibliographystyle{unsrt}

 \end{document}